\documentclass{amsart}
\usepackage{amsmath, amsthm, amsfonts, amssymb, graphicx, hyperref, bm,verbatim,mathrsfs,siunitx}
\usepackage{stmaryrd,enumerate}
\theoremstyle{plain}
\newtheorem{thrm}{Theorem}[section]
\newtheorem{lmm}[thrm]{Lemma}
\newtheorem*{lmm*}{Lemma}
\newtheorem{prpstn}[thrm]{Proposition}
\newtheorem{crllry}[thrm]{Corollary}

\newtheorem*{thrm*}{Theorem}

\newtheorem*{rmk}{Remark}

\newtheorem{cnjctr}[thrm]{Conjecture}

\newtheorem*{mainques}{Main Question}
\numberwithin{equation}{section}
\DeclareMathOperator*{\tr}{tr}

\renewcommand{\tilde}{\widetilde}
\newcommand{\cW}{\mathcal{W}}
\newcommand{\ZZ}{\mathbb{Z}}
\newcommand{\RR}{\mathbb{R}}

\usepackage{color}

\title{New large value estimates for Dirichlet polynomials}
\author{Larry Guth and James Maynard}

\begin{document}
\begin{abstract}
We prove new bounds for how often Dirichlet polynomials can take large values. This gives improved estimates for a Dirichlet polynomial of length $N$ taking values of size close to $N^{3/4}$, which is the critical situation for several estimates in analytic number theory connected to prime numbers and the Riemann zeta function. As a consequence, we deduce a zero density estimate $N(\sigma,T)\le T^{30(1-\sigma)/13+o(1)}$ and asymptotics for primes in short intervals of length $x^{17/30+o(1)}$.
\end{abstract}
\maketitle

%
%
%
%

\section{Introduction}

In this paper we prove new bounds for the frequency of large values of Dirichlet polynomials. This gives improved estimates for a Dirichlet polynomial of length $N$ taking values of size close to $N^{3/4}$.  Our main result is the following.

%
%

\begin{thrm}[Large values estimate]\label{thrm:LargeValues}
Suppose $(b_n)$ is a sequence of complex numbers with $|b_n| \le 1$,  and $(t_r)_{r\le R}$ is a sequence of $1$-separated points in $[0,T]$ such that
\[
\Bigl|\sum_{n= N}^{2N}b_n n^{it_r}\Bigr|\ge V
\]

for all $r\le R$. Then we have
\[
R\le T^{o(1)}\Bigl(N^2V^{-2}+N^{18/5}V^{-4}+TN^{12/5}V^{-4}\Bigr).
\]

\end{thrm}

%
%

The bound of Theorem \ref{thrm:LargeValues} can be compared with the bound

\begin{equation}
R\le T^{o(1)}\Bigl(N^2V^{-2}+T\min(NV^{-2},N^4 V^{-6})\Bigr)
\label{eq:ClassicalLargeValue}
\end{equation}

coming from combining the classical Mean Value Theorem for Dirichlet polynomials and the Montgomery-Hal\'asz-Huxley large values estimate. Together these previous results give bounds at least as strong as Theorem \ref{thrm:LargeValues} when $V$ is smaller than $N^{7/10}$ or bigger than $N^{8/10}$, but Theorem \ref{thrm:LargeValues} gives a stronger bound when $N^{7/10+\epsilon}<V<N^{8/10-\epsilon}$ and $N\le T^{5/6-\epsilon}$. There are various improvements of the large values estimate which will supersede ours when $V$ is a bit larger than $N^{3/4}$ (see \cite[Chapter 11]{Iv}, for example), but Theorem \ref{thrm:LargeValues} represents the first substantive improvement on the bound \eqref{eq:ClassicalLargeValue} when $N^{7/10+\epsilon}\le V\le N^{3/4}$ (when the Mean Value Theorem gave the previous best bounds) or on bounds when $V$ is close to $N^{3/4}$ (where the previous best bounds gave only very small improvements on \eqref{eq:ClassicalLargeValue}).

The key interest in this result is that for many applications in analytic number theory, it is the case when $V\approx N^{3/4}$ which is the critical limiting scenario. In this situation the Mean Value Theorem and Large Values Estimate both give a bound of roughly $N^{1/2}+TN^{-1/2}$ whereas Theorem \ref{thrm:LargeValues} gives roughly $N^{3/5}+TN^{-3/5}$. Therefore we obtain an improvement for $N$ smaller than $T^{10/11}$, and correspondingly we expect Theorem \ref{thrm:LargeValues} to lead to a quantitative improvement to any result where this covers the limiting situation.

One well-studied situation where the limiting case is  improved by Theorem \ref{thrm:LargeValues} is zero-density estimates for the Riemann Zeta function $\zeta(s)$. Let $N(\sigma,  T)$ be the number of zeroes of $\zeta(s)$ in the rectangle $\Re(s) \ge \sigma$ and $|\Im(s) | \le T$.      After early work by Carlson \cite{Ca},  Ingham \cite{I} proved the bound 

\begin{equation}
N(\sigma,  T) \le T^{\frac{3 (1 - \sigma)}{2 - \sigma} + o(1)},
\label{eq:Ingham}
\end{equation}

ultimately relying on the Mean Value Theorem for Dirichlet polynomials. Huxley \cite{Hu},  building on work of Montgomery \cite{M3} and Hal\'asz \cite{H} and ultimately relying on the Montgomery-Hal\'asz large values estimate, proved the bound

\begin{equation}
 N(\sigma, T) \le T^{\frac{3(1 - \sigma)}{3 \sigma - 1} + o(1)}. 
\label{eq:Huxley}
\end{equation}

This improves on \eqref{eq:Ingham} for $\sigma > 3/4$ (corresponding to when the term $N^4V^{-6}$ is smaller than $NV^{-2}$ for $V=N^\sigma$ in $\min(NV^{-2},N^4V^{-6})$ in \eqref{eq:ClassicalLargeValue}).

When $\sigma\approx 3/4$ the bounds \eqref{eq:Ingham} and \eqref{eq:Huxley} coincide, and the critical situation turns out to be related to values of size $V=N^{3/4}$ of Dirichlet polynomials of length $N=T^{4/5}$, where both estimates give $R\le T^{3/5+o(1)}$. In this situation Theorem \ref{thrm:LargeValues} gives an improved estimate of $R\le T^{13/25+o(1)}$. Incorporating Theorem \ref{thrm:LargeValues} into the zero density machinery gives the following result.

%
%

\begin{thrm}[Zero density estimate]\label{thrm:ZeroDensity}
Let $N(\sigma,T)$ denote the number of zeros $\rho$ of $\zeta(s)$ with $\Re(\rho)\ge \sigma$ and $|\Im(\rho)|\le T$. Then we have

\[
N(\sigma,T)\le T^{15(1-\sigma)/(3+5\sigma)+o(1)}.
\]

\end{thrm}

Combining this with Ingham's estimate when $\sigma\le 7/10$, we obtain 

\begin{equation}
N(\sigma,T)\le T^{30(1-\sigma)/13+o(1)}.
\label{eq:ZeroDensity}
\end{equation}

The exponent $30/13$ improves on the previous exponent of $12/5$ due to Huxley \cite{Hu}. This has the following corollaries for the distribution of primes in short intervals.

%
%

\begin{crllry}[Count of primes in short intervals]\label{crllry:All}
Let $y\in[x^{17/30+\epsilon},x^{0.99}]$. Then we have

\[
\pi(x+y)-\pi(x)=\frac{y}{\log{x}}+O_\epsilon\Bigl(y \exp(-\sqrt[4]{\log{x}})\Bigr).
\]

\end{crllry}

The exponent $\frac{17}{30}$ improves on the previous exponent $\frac{7}{12}$ due to Huxley \cite{Hu}.  

%
%

\begin{crllry}[Count of primes in `almost-all' short intervals]\label{crllry:AlmostAll}
Let $y\in [X^{2/15+\epsilon},X^{0.99}]$. Then for all but $O(X\exp(-\sqrt[4]{\log{x}}))$ choices of $x\in [X,2X]\cap\mathbb{N}$ we have

\[
\pi(x+y)-\pi(x)=\frac{y}{\log{x}}+O_\epsilon\Bigl(y \exp(-\sqrt[4]{\log{x}})\Bigr).
\]

\end{crllry}

The exponent $\frac{2}{15}$ improves on the previous exponent $\frac{1}{6}$ due to Huxley \cite{Hu}.

We expect there to be various further applications of Theorem \ref{thrm:LargeValues} (and the underlying ideas) to improving quantitative estimates related to the primes and similar objects.

%
%

\subsection{Background}

Given an integer $N\in \mathbb{N}$ and a sequence $(b_n)_{N<n\le 2N}$, a Dirichlet polynomial is a trigonometric sum of the form

\begin{equation} \label{defdir} D(t) = \sum_{N<n\le 2N} b_n e^{i t \log n}. \end{equation} 

\noindent We say such a polynomial has length $N$, and we study the question how often a Dirichlet polynomial can be large in the interval $[0, T]$.  The precise question can be formulated as follows.

\begin{mainques} Suppose that $D(t)$ is a Dirichlet polynomial of the form \eqref{defdir} with $|b_n| \le 1$ for all $n$.  Suppose that $W \subset [0, T]$ is a 1-separated set of points $t$ where $|D(t)| > N^\sigma$.  What is the largest possible cardinality of $W$, in terms of $N$, $T$, and $\sigma$?
\end{mainques}

There are many examples of Dirichlet polynomials where $|W| \gg  N^{2 - 2 \sigma}$ when $\sigma\in(1/2,1)$ and $N\le T$.\footnote{For example, we can form such a Dirichlet polynomial as follows. If $w$ is a smooth 1-bounded function supported on $[1,2]$, $\epsilon>0$ is a sufficiently small fixed constant and $M^{\epsilon}<H<M/T^{\epsilon}$, then it follows from Poisson summation and repeated integration by parts that
\[
\Bigl|\sum_{h}w\Bigl(\frac{h}{H}\Bigr)(M+h)^{-it}\Bigr|= \begin{cases}
\gg H,\qquad &\text{if }|t|\le M/(10H),\\
\ll \epsilon^2M^2/(t^2H),&\text{if } M/(\epsilon H)\ll |t|\le M.
\end{cases}
\]
Therefore (taking $H=N^\sigma$ and $H$-separated choices of $M\in [N,2N]$) if $t_1,\cdots,t_J\in [T,T+N]$ are $J=\epsilon N^{1-\sigma}+O(1)$ points which are $\epsilon^{-1}N^{1-\sigma}$-separated, the Dirichlet polynomial
\[
D(t):=\sum_{j=1}^J \sum_{h}w\Bigl(\frac{h}{ N^\sigma}\Bigr)\Bigl(N+(j-1)\Bigl\lfloor \frac{N}{J}\Bigr\rfloor+h\Bigr)^{i(t_j-t)}
\]
will take a value of size $\gg N^\sigma$ on an interval of length $\gg N^{1-\sigma}$ around each of the $\gg \epsilon N^{1-\sigma}$ points $t_j$. Thus we can find $\gg_\epsilon N^{2-2\sigma}$ 1-separated points where it takes a value of size $\gg N^\sigma$.}

 Montgomery's large value conjecture \cite[page 142, Conjecture 2]{M2} predicts that under some natural conditions this lower bound is essentially tight. Bourgain \cite{Bour} gave a counterexample to an earlier (stronger) formulation given in terms of the $\ell^2$ norm of the coefficients $b_n$.

%
%

\begin{cnjctr}[Montgomery's large value conjecture] \label{cnjctr:Montgomery} 
Let $\sigma > 1/2$ and $D(t) = \sum_{N<n\le 2N} b_n e^{i t \log n}$ with $|b_n| \le 1$.  Suppose $W \subset [0,T] $ is a 1-separated set such that $|D(t)| > N^\sigma$ for $t\in W$.   Then there is a constant $C(\sigma)$ such that

\[
|W| \le C(\sigma) T^{o(1)} N^{2 - 2 \sigma}.
\]

\end{cnjctr}

%
%

Using a simple orthogonality argument,  it has long been known that for any $T \ge N$, 

\begin{equation} \label{basicorth} |W| = O( T N^{1 - 2 \sigma}),
\end{equation} 

which corresponds to the term $NV^{-2}$ in $\min(NV^{-2},N^4V^{-6})$ in \eqref{eq:ClassicalLargeValue}.

This basic estimate has been improved in several regimes.    An integer power of a Dirichlet polynomial is a Dirichlet polynomial, and applying (\ref{basicorth}) to powers of $D$ gives improved bounds when $N$ is fairly small compared with $T$.  For $\sigma \le 3/4$,  the best previous bounds on our main question came from this approach.  In particular, if $\sigma \le 3/4$ and if $N\in[T^{2/3}, T]$,  then (\ref{basicorth}) represented the best known bound.  

In the late 60s,  Montgomery \cite{M}, building on ideas of Hal\'asz \cite{H} and Tur\'an \cite{HT},  showed that Conjecture \ref{cnjctr:Montgomery} is true if $\sigma$ is sufficiently large. Indeed, \eqref{eq:ClassicalLargeValue} gives Conjecture \ref{cnjctr:Montgomery} when $N^\sigma>N^{1/2}T^{1/4}$. The large value method gives very strong information for large $\sigma$,  but it gives no information if $\sigma \le 3/4$ as Montgomery explains in \cite[page 141]{M2}. For $\sigma$ closer to $1$, there have been several refinements of the underlying ideas (see, for example, \cite{B,HB3,J}).

If one knows some more structure about the set of large values of a Dirichlet polynomial, then one can hope to have improved bounds. For example, another important result is Heath-Brown's work about the behavior of Dirichlet polynomials on difference sets:

%
%

\begin{thrm}[Heath-Brown, \cite{HB}]\label{thrm:HeathBrown}Suppose that $\mathcal{T}$ is a 1-separated set of points in an interval of length $T$.   Let $|a_n|\le T^{o(1)}$ be a complex sequence. Then

\[
\sum_{t_1,t_2\in \mathcal{T}}\Bigl|\sum_{n=N}^{2N} a_n n^{i(t_1-t_2)}\Bigr|^2\le T^{o(1)}\Bigl( |\mathcal{T}|^2N+|\mathcal{T}|N^2+|\mathcal{T}|^{5/4}T^{1/2}N\Bigr).
\]

\end{thrm}

%
%

In the range $N \in[T^{2/3}, T ]$, the last term can be ignored and the right-hand side is essentially sharp\footnote{If $N\ge T^{2/3}$ then $|\mathcal{T}|^{5/4}T^{1/2}N\le |\mathcal{T}|^{5/4}N^{7/4}=(|\mathcal{T}|^2N)^{1/4}(|\mathcal{T}|N^2)^{3/4}\le |\mathcal{T}|^2N+|\mathcal{T}|N^2$.}.   Theorem \ref{thrm:HeathBrown} gives strong information about our main question if the set $W$ has a lot of additive structure,  such as arithmetic progressions.

We will make this precise using the idea of additive energy,  which we define as follows.  For a finite set $W$,  let

\begin{equation}
E(W):=\#\{w_1,w_2,w_3,w_4\in W:\,|w_1+w_2-w_3-w_4|\le 1\}.
\label{eq:EnergyDef}
\end{equation}

A simple consequence of Theorem \ref{thrm:HeathBrown} is the following (we give a proof, as well as refined bounds, in Section \ref{sec:Energy}).

%
%

\begin{lmm}\label{lmm:BasicEnergy}
Let $N\in [T^{2/3},T]$, $\sigma > 1/2$ and $D(t) = \sum_{N<n\le 2N} b_n n^{i t }$ with $|b_n| \le 1$. Suppose $W \subset [0,T] $ is a 1-separated set such that $|D(t)| > N^\sigma$ for $t\in W$. Then

\[
E(W)\le |W|^3 N^{1-2\sigma+o(1)}+|W|^2 N^{2-2\sigma+o(1)}.
\]

\end{lmm}

If $E(W)$ is very large, say $E(W)> |W|^3T^{-o(1)}$, then $E(W)$ is larger than the first term on the right hand side, so it must be bounded by the second term. This implies Conjecture \ref{cnjctr:Montgomery} for Dirichlet polynomials $D$ with $E(W)\ge |W|^3 T^{-o(1)}$. Thus we can obtain improved bounds for Dirichlet polynomials whose large value set has a lot of additive structure.

We introduce a new method that gives good bounds for Dirichlet polynomials on sets of small energy.   We outline our method in the next section.

\begin{rmk}
Many earlier results on large values of Dirichlet polynomials (such as those stemming from the Mean Value Theorem or the Montgomery-Huxley-Hal\'asz large values estimate) allowed one to relax the $\ell^\infty$ constraint $|b_n|\le 1$ on the coefficients to a weaker $\ell^2$ constraint $(\sum_{n=N}^{2N}|b_n|^2)^{1/2}=O(N^{1/2})$. Our argument crucially relies on the stronger $\ell^\infty$ assumption to bound the additive energy $E(W)$ efficiently, and we do not know how to obtain Theorem \ref{thrm:LargeValues} if we only have $\ell^2$ bounds on the coefficients. The remark after Lemma \ref{lmm:Energy1} highlights where we use this assumption.
\end{rmk}

%
%

\subsection{Notations and conventions} \label{subsecdeltalessapprox}

We write $A \ll B$ to mean that $ |A| \le C B$ for an absolute constant $C$, and $A \ll_z B$ to denote that the constant $C$ may depend on the parameter $z$. We write $A\asymp B$ to mean $A\ll B$ and $B\ll A$ both hold, and $A\sim B$ to mean $B<A\le 2B$. Similarly,  we write $A \lessapprox B$ to mean that for any $\epsilon > 0$, there is a constant $C(\epsilon)>0$ depending only on $\epsilon$ such that $|A| \le C(\epsilon) T^{\epsilon} B$ for all large $T$, and $A \lessapprox_z B$ to denote a dependency on a parameter $z$. Asymptotic quantities such as $o(1)$ are interpreted as $T\rightarrow \infty$.

We use $e(x):=e^{2\pi i x}$ to denote the complex exponential. Summations will run over integers unless specified otherwise.

%
%

\subsection{Acknowledgements.} LG would like to thank Yuqiu Fu for many interesting discussions about this problem.  
LG is supported by a Simons Investigator Award. JM is supported  by the European Research Council (ERC) under the European Union’s
 Horizon 2020 research and innovation programme (grant agreement No 851318). We would like to thank the anonymous referees for a number of helpful comments and corrections. For the purposes of open access, the authors have applied a CC-BY license to any author accepted manuscript arising from this submission.

%
%
%
%

\section{Sketch outline}

For a heuristic description of our argument, let us consider a critical case of previous zero density estimates, where $\sigma = 3/4$ and we wish to improve upon the bound $|W|\lessapprox TN^{-1/2}$ for any set $W$ of separated points in an interval of length $T$ such that $|D(t)|>N^{3/4}$ when $t\in W$. We focus on the situation when $T=N^{1+\delta}$ for some small constant $\delta>0$ (a corresponding improvement for larger values of $T$ then follows by subdivision), and suppress many technical details (such as the presence of smoothing in most summations) for exposition. 

Let $M$ be the $|W|\times N$ matrix with entries

\[
M_{t,n}=n^{it}
\]

for $t\in W$ and $n\sim N$. We see that if $\mathbf{b}=(b_n)_{n\sim N}$ then

\[
D(t)=\sum_{n}b_n n^{it}=(M\mathbf{b} )_t,
\]

and so we have (recalling $|b_n|\le 1$)

\[
N^{3/2}|W|=N^{2\sigma} |W|<\sum_{t\in W}|D(t)|^2=\|M\mathbf{b}\|_2^2\le \|M\|^2 \|\mathbf{b}\|_2^2\le N\|M\|^2,
\]

where $\|M\|$ is the matrix (operator) norm of $M$. Thus we wish to improve upon the bound $\|M\|\lessapprox T^{1/2}$ which follows from the Mean Value Theorem (or the Large Values Estimate), and we might hope that $\|M\|\approx |N|^{1/2}$. We note that

\[
\|M\|= s_1(M)
\]

where $s_1(M)$ is the largest singular value of $M$, which is also the square-root of the largest eigenvalue of $M^*M$. Substituting this in above, we find

\[
|W|\le N^{-1/2}s_1(M)^2.
\]

A simple bound for $s_1(M)$ is then given by the trace of powers of $M^*M$; for any $r\in \mathbb{N}$

\[
s_1(M)\le \tr((M^*M)^r)^{1/2r}.
\]

We might hope $\tr((M^*M)^r)\approx |W|N^{r}$, which would imply $s_1(M)\le N^{1/2}|W|^{1/2r}$. 
Unfortunately it appears to be very difficult to estimate these traces accurately for large $r$, and we would only improve upon the bound $|W|\lessapprox TN^{-1/2}$ if $(TN^{-1/2})^{1/r}\le TN^{-1}$, which requires $r$ to be large. Nevertheless, a variant of this bound for $r=3$ (which reduces the contribution from some of the trivial terms) allows one to show a bound similar to

\[
s_1(M)^6\lessapprox \Bigl|\sum_{\substack{t_1,t_2,t_3\in W\\ |t_j-t_k|>T^\epsilon\,\forall j\ne k}}\sum_{n_1,n_2,n_3\sim N}n_1^{i(t_1-t_2)}n_2^{i(t_2-t_3)}n_3^{i(t_3-t_1)}\Bigr|.
\]

The right hand side would be $\tr((M^*M)^3)$ if we didn't have the lower bounds on $|t_i-t_j|$ (and so we avoid the $t_1=t_2=t_3$ terms). To improve upon the bound $|W|\lessapprox TN^{-1/2}$ we wish to show the right hand side is a bit smaller than $T^3$.

The natural approach to estimating the right hand side would be to apply the approximate functional equation (or Poisson summation and stationary phase) to each of the inner sums, which would yield a bound of roughly

\[
\frac{N^3}{T^{3/2}}\Bigl|\sum_{t_1,t_2,t_3\in W}\sum_{m_1,m_2,m_3\sim T/N}m_1^{i(t_1-t_2)}m_2^{i(t_2-t_3)}m_3^{i(t_3-t_1)}c_{t_1-t_2}c_{t_2-t_3}c_{t_3-t_1}\Bigr|
\]

for some coefficients $c_{t}$ of size 1. Unfortunately the $c_{t_i-t_j}$ coefficients link the variables $t_1,t_2$ and $t_3$, and this makes it difficult to show any cancellation over these summations. Without any cancellation in the $t_1,t_2,t_3$ sums we would be limited to a bound of $|W|^3N^{3/2}\approx T^3$ even if we obtained square-root cancellation in the sums over $m_1,m_2,m_3$, and so this would fail to give the desired improvement on $T^3$ (but could potentially give results if $\sigma>3/4$).

A key observation is that it can be beneficial to avoid the use of stationary phase. Applying Poisson summation without simplifying the Fourier integrals yields a bound of roughly

\begin{align*}
&N^3\sum_{t_1,t_2,t_3\in W}\sum_{|m_1|,|m_2|,|m_3|\sim T/N}\int_{[1,2]^3}e^{-2\pi i N\mathbf{m}\cdot \mathbf{u}}\Bigl(\frac{u_1}{u_3}\Bigr)^{it_1}\Bigl(\frac{u_2}{u_1}\Bigr)^{it_2}\Bigl(\frac{u_3}{u_1}\Bigr)^{it_3}d\mathbf{u}\\
 &\le N^3\sum_{|\mathbf{m}|\sim T/N}\Bigl|\int_{[1,2]^3}e^{-2\pi i N\mathbf{m}\cdot \mathbf{u}} R\Bigl(\frac{u_1}{u_3}\Bigr)R\Bigl(\frac{u_2}{u_1}\Bigr)R\Bigl(\frac{u_3}{u_2}\Bigr)d\mathbf{u}\Bigr|,
\end{align*}

where $R(x):=\sum_{t \in W}|x|^{it}$. By not simplifying the integrals over $\mathbf{u}$ directly we have given up a factor of $T^{3/2}$, but now the variables $t_1,t_2,t_3$ are nicely separated from one another, and we can hope to show cancellation in these sums by showing the $R$ function exhibits cancellation. However, even square-root cancellation in $R(x)$ would only win back a factor of $|W|^{3/2}$ which is less than the $T^{3/2}$ factor we gave up by not applying stationary phase. Therefore we need to exploit simultaneous cancellation in the integral over $\mathbf{u}$ and the $R$ functions.

The next important observation is that the arguments of the $R$ functions only lie in a two dimensional subvariety $z_1z_2z_3=1$. If we change variables to $v_1=u_1/u_3$ and $v_2=u_2/u_3$ then the integral above is roughly

\[
\int_{[1/2,2]^2}\Bigl(\int_{[1,2]} e^{-2\pi i N (m_1 v_1+ m_2 v_2+m_3)u_3}du_3\Bigr)R(v_1)R\Bigl(\frac{v_2}{v_1}\Bigr)R\Bigl(\frac{1}{v_2}\Bigr) dv_1 dv_2.
 \]
 
There is a large amount of cancellation in the inner integral unless $m_1 v_1+ m_2 v_2+m_3\approx 0$, and this allows us to win back a factor of $T$. We are then reduced to bounding expressions of the form

\begin{equation}
 \frac{N^3}{T}\sum_{|\mathbf{m}|\sim T/N}\Bigl| \int_{\substack{(v_1,v_2)\in [1/2,2]^2\\ m_1 v_1+ m_2 v_2+m_3=0}} R(v_1)R\Bigl(\frac{v_2}{v_1}\Bigr)R\Bigl(\frac{1}{v_2}\Bigr) dv_1\Bigr|.
\label{eq:KeySum}
\end{equation}

If we had uniform square-root cancellation in the $R$ functions, we would now get a bound of $T^2 |W|^{3/2}$, which comfortably beats the desired bound of $T^3$ when $|W|\approx TN^{-1/2}$, and so we would get a corresponding improvement to the large value estimates and zero density results.

Of course, we cannot expect to prove uniform square-root cancellation in the $R$ function for arbitrary sets $W$. Nevertheless, one can show that $R$ does exhibit cancellation on average; $\|R(v)\|_{L^2([1/2,2])}^2\lessapprox |W|$ and $\|R(v)\|_{L^4([1/2,2])}^4\lessapprox E(W)$, where $E(W)$ from \eqref{eq:EnergyDef} counts approximate additive quadruples in $W$. Treating the $\mathbf{m}$ summation trivially, this would give a bound

\begin{equation}
T^2|W|^{1/2}E(W)^{1/2},
\label{eq:CrudeEnergy}
\end{equation}

which would beat the target of $T^3$ if $E(W)$ is a bit smaller than $T^2/|W|\approx N^2 |W|^3/T^2$. Thus we obtain good bounds whenever $E(W)$ is small. In particular, when $T\approx N^{1+\delta}$ with $\delta$ small, this means that we can handle any set $W$ except those for which the additive energy is close to the maximal possible value of $|W|^3$.

We would like to complement this with an argument for when $E(W)$ is large. As mentioned in the introduction, Heath-Brown's result becomes useful in this situation. In this case, Lemma \ref{lmm:BasicEnergy} implies that $E(W)\ll |W|^2N^{1/2}\approx |W|^3N/T$, which isn't quite strong enough to cover all ranges. A refinement of this lemma can improve this to $E(W)\ll |W|^3N^2/T^2$ (at least when $\delta$ is small and $|W|\approx TN^{-1/2}$), which then allows us to get a small improvement in all cases except when $E(W)\approx |W|^3N^2/T^2$.

To overcome this final obstacle when $E(W)\approx |W|^3N^2/T^2$, we go back to \eqref{eq:KeySum}, and exploit the averaging over $\mathbf{m}$. After a change of variables we need to consider

\[
\frac{N^3}{T} \sum_{|m_1|,|m_2|,|m_3|\sim T/N}\int_{v_1 \in [1/2,2]}  \Big| R ( v_1 )  R \Big( \frac{ m_1 v_1 + m_3}{m_2 v_1} \Big)  R \Big( \frac{m_1 v_1 + m_3}{m_2} \Big) \Big| dv_1.
\]

The only way the bound above could be tight is if there is a sparse set $U$ (of measure roughly $TN^{-3/2}$) on which $R(u)$ is large (taking values of size roughly $N^{1/2}$), and such that for many $u\in U$ and many $|m_1|,|m_2|,|m_3|\sim T/N$ we have $(m_1u+m_3)/m_2\in U$ and  $(m_1u+m_3)/(m_2u)\in U$. We show that there cannot be a small set $U$ which has this property of being approximately closed under many affine transformations $u\mapsto (m_1u+m_2)/m_3$, and this ultimately leads to an improvement of \eqref{eq:CrudeEnergy} to roughly

\[
T N |W|^{1/2}E(W)^{1/2}.
\]

This now gives an improvement on the desired bound of $T^3$ whenever $E(W)$ is smaller than $|W|^3$, and so the bounds obtained on $E(W)$ stemming from Heath-Brown's result above are now sufficient to give an improvement for any size of $E(W)$. Therefore we obtain an improvement on the bound $|W|\lessapprox TN^{-1/2}$ for arbitrary $W$.

%
%
%
%

\section{Reduction to Main Proposition}\label{secredmain}

In this section we reduce Theorem \ref{thrm:LargeValues} to a similar but technically more convenient proposition where we have inserted a smoothing in the $n$ variable (which aids the later Fourier analysis) and used subdivision to specialize to the case $N=T^{4/5}$ (which is when the final two terms in Theorem \ref{thrm:LargeValues} coincide). Throughout the rest of the paper we fix a smooth function $w:\mathbb{R}\rightarrow\mathbb{R}_{\ge 0}$ supported on $[1,2]$ with $\|w^{(j)}\|_\infty=O_j(1)$ for all $j\in \mathbb{Z}_{\ge 0}$ and with $w(t)=1$ for $t\in [6/5,9/5]$.

%
%

\begin{prpstn}\label{prpstn:KeyProp}
Let $\sigma\in [7/10,8/10]$ and $\epsilon>0$. Let $b_n$ be a sequence of complex numbers with $|b_n| \le 1$ and $W$ be a set of $T^\epsilon$-separated points in an interval of length $T=N^{6/5}$ such that

\[
\Bigl|\sum_n w\Bigl(\frac{n}{N}\Bigr)b_n n^{it}\Bigr|\ge N^\sigma
\]

for all $t\in W$. Then we have

\[
|W|\le TN^{(12-20\sigma)/5+o_\epsilon(1)}.
\]

\end{prpstn}

\begin{proof}[Proof of Theorem \ref{thrm:LargeValues} assuming Proposition \ref{prpstn:KeyProp}]
As mentioned in the introduction, the result follows from \eqref{eq:ClassicalLargeValue} if $V\le N^{7/10+o(1)}$ or if $V\ge N^{8/10-o(1)}$, so we may assume $V\in [4N^{7/10},N^{8/10}]$, in which case $N^2V^{-2}\le N^{18/5}V^{-4}$. Similarly, Theorem \ref{thrm:LargeValues} follows from \eqref{eq:ClassicalLargeValue} if $N\ge T$ since then $R\le T^{o(1)}N^2V^{-2}$, so we may assume $N<T$. By splitting $D(t)$ into 3 separate pieces (and using the triangle bound), we see that it suffices to show the result when $b_n=0$ unless $n\in [6N/5,9N/5]$. Since the function $w$ is 1 on $[6/5,9/5]$, we then have that $b_n=b_n w(n/N)$, so we may insert the weight $w(n/N)$. Finally, having inserted the smooth weights we now relax the vanishing condition on the $b_n$. Thus, letting $V=N^\sigma$, it suffices to show that whenever $\sigma \in [7/10,8/10]$, $N\ll T$, $(a_n)$ is a 1-bounded complex sequence and $\cW$ is a set of $1$-separated points such that 

\[
\Bigl|\sum_{n}w\Bigl(\frac{n}{N}\Bigr)a_n n^{it}\Bigr|\ge N^{\sigma}
\]

for each $t\in \cW$, we have 

\[
|\cW|\le T^{o(1)}(N^{18/5-4\sigma}+TN^{12/5-4\sigma}).
\]

We now fix $\eta>0$ and choose $\cW' \subset \cW$ so that $\cW'$ is $T^{\eta}$-separated and $|\cW'| \ge |\cW|/T^{\eta}$ (this can be achieved by picking the smallest element of $\cW$ and then repeatedly choosing the next smallest element which is at least $T^{\eta}$ away from all picked elements).   If $T \le N^{6/5}$, then we apply Proposition \ref{prpstn:KeyProp} to bound $\cW'$ directly (taking $\epsilon=\eta/2$ so that $\cW'$ is $N^{6\epsilon/5}$-separated since $M\ll T$), which implies that 

\[
|\cW|\le T^{\eta}|\cW'|\le T^\eta N^{(18-20\sigma)/5+o_\eta(1)}.
\]

Letting $\eta\rightarrow 0$ sufficiently slowly then gives the result in this case.  If instead $T > N^{6/5}$,  then we divide $\cW'$ into $\lceil T/ N^{6/5}\rceil$ subsets $\cW_j'$  each supported on an interval of length $N^{6/5}$,  and we apply Proposition \ref{prpstn:KeyProp} to bound each $\cW_j'$ separately. This gives

\[
|\cW|\le T^\eta\sum_{j\le \lceil T/N^{6/5} \rceil }|\cW_j'|\le T^{1+\eta} N^{(12-20\sigma)/5+o_\eta(1)}.
\]

Letting $\eta\rightarrow 0$ sufficiently slowly then gives the result in this case too.
\end{proof}

\section{The  matrix \texorpdfstring{$M_W$}{MW} and its singular values}

Now we begin to work on the proof of Proposition \ref{prpstn:KeyProp}.  We will work with the smoothed version $D_N(t)$ of $D(t)$ from Proposition \ref{prpstn:KeyProp}

\begin{equation}
D_N(t) :=\sum_{n} w\Bigl(\frac{n}{N}\Bigr)b_n n^{it}, 
\label{eq:DNDef}
\end{equation}

\noindent where $w$ is the smooth bump supported on $[1,2]$ defined in Section \ref{secredmain}. Similarly, given a set $W\subseteq\mathbb{R}$, let $M_W$ be the $|W|\times N$ matrix with smoothed entries
\begin{equation}
(M_W)_{t,n}=w(n/N)n^{it}, 
\label{eq:MWDef}
\end{equation}

\noindent where $t \in W$ and $n \sim N$. 

%
%

\begin{lmm}[Large values of Dirichlet polynomials controlled by singular values]\label{lmm:SpectralBound}
Let $M_W$ be the matrix defined in \eqref{eq:MWDef}, and $s_1(M_W)$ its largest singular value. If $|D_N(t)| \ge N^\sigma$ on $W$ and if $|b_n| \le 1$, then we have

\[
|W|\ll N^{1-2\sigma}s_1(M_W)^2.
\]

\end{lmm}
\begin{proof}
Let $\mathbf{b}$ be the vector with components $b_n$.   Then note that for each $t$ in $W$, 

\[ 
D_N(t) =\sum_{n} w(n/N)b_n n^{it}= (M_W \mathbf{b})_t.  
\]

Therefore we can relate the behavior of $D_N$ on $W$ (for arbitrary $\mathbf{b}$) to properties of the matrix $M_W$, in particular its singular values.  We write $s_j(M_W)$ for the $j^{th}$ singular value of $M_W$,  with the convention that $s_1(M_W) \ge s_2(M_W) \ge ...\ge s_k(M_W)$ and $k=\min(|W|,N)$ is the number of singular values.   Let $M_W$ have singular value decomposition $M_W=U\Sigma V$, so that $\Sigma$ is a rectangular matrix with $\Sigma_{ii}=s_i(M_W)$ and $\Sigma_{ij}=0$ if $i\ne j$, and $U$, $V$ are unitary matrices.

If $|D_N(t)| \ge N^\sigma$ on $W$,  then we see

\begin{align*}
|W| N^{2 \sigma} \le \sum_{t\in W}|D_N(t)|^2=(M_W\mathbf{b})^* M_W \mathbf{b}&=(V\mathbf{b})^* \Sigma^2 V\mathbf{b}\\
&\le s_1(M_W)^2 \| V\mathbf{b} \|_{\ell^2}^2\\
&= s_1(M_W)^2 \| \mathbf{b} \|_{\ell^2}^2.
\end{align*}
Finally, if $|b_n|\le 1$ then $\| \mathbf{b} \|_{\ell^2}^2\ll N$. Substituting this into the expression above and rearranging now gives the result.
\end{proof}

Now $s_1(M_W)$ is equal to the square root of the largest eigenvalue value of the $|W|\times |W|$ matrix $M_W M_W^*$, with entries

\[
(M_W M_W^*)_{t_1,t_2}=\sum_{n} w \left( \frac{n}{N} \right)^2 n^{i(t_1-t_2)}.
\]

A simple bound for $s_1(M_W)$ is therefore to use the trace: for any integer $r\ge 1$ we have

\[
s_1(M_W)^2= s_1(M_W M_W^*)\le \Bigl(\sum_{j=1}^ks_j(M_W M_W^*)^r\Bigr)^{1/r}=\tr((M_W M_W^*)^r)^{1/r}.
\]

One might guess that $s_j(M_W)\lessapprox N^{1/2}$ for all $j$, in which case we would have 

\[
\tr((M_W M_W^*)^r)^{1/r}\lessapprox |W|^{1/r}N.
\]

 If one could establish such a sharp bound on $\tr((M_W^*M_W)^r)$ for large $r$, this would give Conjecture \ref{cnjctr:Montgomery}. Unfortunately we do not know how to obtain good bounds when $r\ge 4$, so we work with $r=3$. In this case, even a sharp bound 
 
 \[
 \tr((M_W M_W^*)^3)^{1/3}\lessapprox |W|^{1/3}N
 \]
 
  would only yield $|W|\lessapprox N^{3-3\sigma}$, which is worse than the bounds established by previous works. To get around this issue, we note that if we are in the extreme scenario when

\[
\frac{1}{k}\sum_{i=1}^{k} s_i(M_W)^6 =\Bigl(\frac{1}{k}\sum_{i=1}^{k} s_i(M_W)^2\Bigr)^3
\]

then in fact we must have that all the singular values are the same, and so 

\[
s_1(M_W)=\tr((M_WM_W^*)^3)^{1/6}k^{-1/6},
\]

 a significant improvement on the bound $\tr((M_WM_W^*)^3)^{1/6}$ we had before. Similarly, we would expect that if $\tr((M_WM_W^*)^3)$ is close to $\tr(M_WM_W^*)^3/k^{2}$, then we would also get an improved bound on $s_1(M_W)$ since most of the contribution to $\tr((M_WM_W^*)^3)$ would be coming from the many other singular values. The following lemma makes this precise, stating that we can essentially replace $\tr((M_WM_W^*)^3)$ with the difference 
 
 \[
 \tr((M_WM_W^*)^3)-\frac{\tr(M_WM_W^*)^3}{k^{2}}
 \]
 
  for the purposes of bounding $s_1(M_W)$.

%
%

\begin{lmm}[Bound for singular values in terms of traces]
Let $A$ be an $m\times n$ complex matrix. Then we have
\[
s_1(A)\le 2\Bigl(\tr((AA^*)^3)-\frac{\tr(AA^*)^3}{m^2}\Bigr)^{1/6}+2\Bigl(\frac{\tr(AA^*)}{m}\Bigr)^{1/2}.
\]
\end{lmm}

\begin{proof}
Recall that $\tr((AA^*)^j)=\sum_{i=1}^m\lambda_i^j$ where $\lambda_1,\dots,\lambda_m$ are the eigenvalues of the $m\times m$ matrix $AA^*$, which are real and non-negative, and that $s_1(A)=\max_i\lambda_i^{1/2}$. We see that it is sufficient to show for any non-negative reals $x_1,\dots x_k$

\begin{equation}
x_1\le 2\Bigl(\sum_{i=1}^k x_i^6-\frac{(\sum_{i=1}^k x_i^2)^3}{k^2}\Bigr)^{1/6}+2\Bigl(\frac{\sum_{i=1}^k x_i^2}{k}\Bigr)^{1/2}.
\label{eq:RealIneq}
\end{equation}

By H\"older's inequality, we have that $\sum_{i=2}^k x_i^6\ge (\sum_{i=2}^k x_i^2)^3/(k-1)^2\ge (\sum_{i=2}^k x_i^2)^3/k^2$. Thus

\begin{align*}
x_1^6=\sum_{i=1}^k x_i^6-\sum_{i=2}^k x_i^6&\le \sum_{i=1}^k x_i^6-\frac{(\sum_{i=2}^k x_i^2)^3}{k^2}\\
&\le \Bigl( \sum_{i=1}^k x_i^6-\frac{(\sum_{i=1}^k x_i^2)^3}{k^2}\Bigr)+3x_1^2\frac{(\sum_{i=1}^k x_i^2)^2}{k^2}\\
&\le \max\Bigl( 4\Bigl(\sum_{i=1}^k x_i^6-\frac{(\sum_{i=1}^k x_i^2)^3}{k^2}\Bigr),\,4x_1^2\frac{(\sum_{i=1}^k x_i^2)^2}{k^2}\Bigr).
\end{align*}

This in turn implies

\[
x_1^6\le \max\Bigl( 4\Bigl(\sum_{i=1}^k x_i^6-\frac{(\sum_{i=1}^k x_i^2)^3}{k^2}\Bigr),\,8\frac{(\sum_{i=1}^k x_i^2)^3}{k^3}\Bigr),
\]

which gives \eqref{eq:RealIneq}.
\end{proof}

Thus we wish to estimate $\tr(M_W M_W^*)$ and $\tr((M_W M_W^*)^3)$. In both cases we expand the trace and use Poisson summation as a first step. In anticipation of this, we introduce the function

\begin{equation}
h_t(u):=w(u)^2u^{it},
\label{eq:htDef}
\end{equation}

which appears in the sums defining the coefficients of $M_WM_W^*$. We first record a basic tail estimate for the Fourier transform $\widehat{h}_t$.

%
%

\begin{lmm}[Non-stationary phase]\label{lmm:Fourier}
Let $h_t(u)=w(u)^2u^{it}$. Then we have
\begin{enumerate}
\item For any integer $j\ge 0$ we have
\[
\widehat{h}_t(\xi)\ll_j (1+|t|)^j/|\xi|^j.
\]
\item For any integer $j\ge 0$ we have
\[
\widehat{h}_t(\xi)\ll_j (1+|\xi|)^j/|t|^j.
\]
\end{enumerate}
\end{lmm}

\begin{proof}
Since $\|w^{(j)}\|_\infty\ll_j 1$ for all $j\ge 0$, we have that $\|h_t^{(j)}\|_\infty\ll_j 1+|t|^j$ for all $j\ge 0$. Thus, by integration by parts (and using that $w$ is compactly supported), we have that

\[
\widehat{h}_t(\xi)=\int e(-\xi u)h_t(u)du=\frac{1}{(2\pi i\xi)^j}\int e(-\xi u)h_t^{(j)}(u)du\ll_j \frac{1+|t|^j}{|\xi|^j}.
\]

Similarly, if $g_\xi(u)=e(-\xi u)w(u)^2$ then $\|g_\xi^{(j)}\|_\infty\ll_j 1+|\xi|^j$, so integration by parts gives

\[
\widehat{h}_t(\xi)=\int g_\xi(u)u^{it}du=\frac{(-1)^j}{(it+1)\cdots (it+j)}\int g_\xi^{(j)}(u)u^{it+j}du\ll_j \frac{1+|\xi|^j}{|t|^j}.\qedhere
\]
\end{proof}

%
%

\begin{lmm}[Hilbert-Schmidt Norm estimate]\label{sumsi2}
 If $W \subset \RR$ is a finite set with $|W| \le N^{O(1)}$, then
\[
\tr(M_W M_W^*) = N |W| \,\| w \|_{L^2}^2 +  O(N^{-100}).
\]
\end{lmm}

\begin{proof}
Expanding the trace, we see that

\[
\tr(M_W M_W^*) =\sum_{t \in W}\sum_{ n } w(n/N)^2 = |W| \sum_{n \in \ZZ} h_0(n/N). 
 \]
 
Because $h_0$ is a smooth compactly supported function, the sum $\sum_n h_0(n/N)$ is very close to the integral $N \int_\RR h_0(u) d \xi = N \| w \|_{L^2}^2$.   We can get a precise estimate using Poisson summation, which gives (separating the term $m=0$)

\[
 \sum_n h_0(n/N) = N \sum_m \widehat{h_0} (N m) = N\widehat{h_0}(0)+O\Bigl(N\sum_{m\ne 0}|\widehat{h_0}(N m)|\Bigr).
  \]
  
The first term on the right hand side is $N\|w\|_{L^2}^2$ and the second term is $O(N^{-100})$ by Lemma \ref{lmm:Fourier}.
\end{proof}

%
%

\begin{lmm}[Expansion of the cubic trace]\label{lmm:TraceExpansion}
Let $W$ be $T^\epsilon$-separated. Then we have

\[
\tr((M_W M_W^*)^3) =N^3|W|\|w\|_{L^2}^6+  \sum_{\substack{m\in \mathbb{Z}^3\setminus\{ 0\} }}I_m+ O_\epsilon(T^{-100}),
\]

where

\[
I_m := N^3 \sum_{t_1,t_2,t_3\in W}\widehat{h}_{t_1-t_2}(m_1N)\widehat{h}_{t_2-t_3}(m_2N)\widehat{h}_{t_3-t_1}(m_3N). 
\]
\end{lmm}
\begin{proof}
First we expand $\tr((M_W M_W^*)^3)$ as the sum $S$, given by

\begin{align*}
S& = \sum_{n_1,n_2,n_3\in \mathbb{Z}}\sum_{t_1,t_2,t_3\in W}w\Bigl(\frac{n_1}{N}\Bigr)^2 w\Bigl(\frac{n_2}{N}\Bigr)^2 w\Bigl(\frac{n_3}{N}\Bigr)^2 n_1^{i(t_1-t_2)}n_2^{i(t_2-t_3)}n_3^{i(t_3-t_1)}\\
&=\sum_{t_1,t_2,t_3\in W} \sum_{n_1,n_2,n_3\in \mathbb{Z}} h_{t_1-t_2}\Bigl(\frac{n_1}{N}\Bigr)h_{t_2-t_3}\Bigl(\frac{n_2}{N}\Bigr)h_{t_3-t_1}\Bigl(\frac{n_3}{N}\Bigr),
\end{align*}

where, as in \eqref{eq:htDef}, we have $ h_t(u)= w(u)^2 u^{it}$. We now perform Poisson summation in $n_1, n_2, n_3$, which gives

\[
S=N^3\sum_{m_1,m_2,m_3\in \mathbb{Z}}\sum_{t_1,t_2,t_3\in W}\widehat{h}_{t_1-t_2}(m_1N)\widehat{h}_{t_2-t_3}(m_2N)\widehat{h}_{t_3-t_1}(m_3N)=\sum_{m\in\mathbb{Z}^3}I_m. 
\]

Finally, we separate the term $m_1=m_2=m_3=0$, which contributes

\[
I_0=N^3\sum_{t_1,t_2,t_3\in W}\widehat{h}_{t_1-t_2}(0)\widehat{h}_{t_2-t_3}(0)\widehat{h}_{t_3-t_1}(0).
\]

Since $W$ is $T^\epsilon$-separated, $\widehat{h}_{t_1-t_2}(0)\ll_\epsilon T^{-200}$ if $t_1\ne t_2$ by Lemma \ref{lmm:Fourier}. Thus the terms in the $I_0$ above are negligible unless $t_1=t_2=t_3$, and so

\[
I_0=N^3\sum_{t\in W}\widehat{h}_0(0)^3+O_\epsilon(T^{-100})=N^3|W|\|w\|_{L^2}^6+O_\epsilon(T^{-100}).
\]

Putting this together gives the result.
\end{proof}

Putting together Lemmas \ref{lmm:SpectralBound}-Lemma \ref{lmm:TraceExpansion}, and noting that the $N^3 |W| \|w\|_{L^2}^3$ term cancels with $\tr(M_W M_W^*)^3/|W|^2$, gives the following.

%
%

\begin{prpstn}\label{prpstn:ImBound}
Let $W$ be $T^\epsilon$-separated, and let $|b_n|\le 1$ be such that $|D_N(t)|>N^\sigma$ for all $t\in W$. Then we have

\[
|W|\ll_\epsilon N^{2-2\sigma}+N^{1-2\sigma}\Bigl(\sum_{\substack{m\in \mathbb{Z}^3\setminus\{0\}}}I_m\Bigr)^{1/3},
\]

where $I_m$ is the quantity defined in Lemma \ref{lmm:TraceExpansion}.
\end{prpstn} 

The first term above corresponds to the best possible estimate is $|W|\ll N^{2-2\sigma}$ of Conjecture \ref{cnjctr:Montgomery}, and so our task is reduced to getting a good bound for the sum of $I_m$.

%
%
%
%

\section{The pieces of the sum \texorpdfstring{$S$}{S}}

Recall from Proposition \ref{prpstn:ImBound}, we have

\[
 |W|^{3} \ll N^{6-6\sigma} + N^{3-6\sigma}\sum_{\substack{m \in \ZZ^3\setminus\{0\}}} I_m,
 \]
 
 where
 
 \[
 I_m = N^3 \sum_{t_1,t_2,t_3\in W}\widehat{h}_{t_1-t_2}(m_1N)\widehat{h}_{t_2-t_3}(m_2N)\widehat{h}_{t_3-t_1}(m_3N). 
 \]
 
To get started, we note a few cases when $|\widehat{h}_{t_1 - t_2}(mN)|$ is easy to understand via Lemma \ref{lmm:Fourier}.   Since $W$ is $T^\epsilon$-separated,  we see that if $t_1 \ne t_2$, by Lemma \ref{lmm:Fourier} we have

\begin{equation}\label{tdiffm0} 
| \widehat{h}_{t_1-t_2}(0) | \ll_\epsilon T^{-100}.  
\end{equation}

On the other hand,  if $t_1 = t_2$,  then we have

\begin{equation}
\label{tsamem0} 
\widehat{h}_{t_1 - t_2}(0) = \widehat{h}_0(0) = \int w^2(u) du \asymp 1.
\end{equation}

If $t_1 = t_2$ but $m \not= 0$,  then by Lemma \ref{lmm:Fourier} we have

\begin{equation}
\label{tsamemn0} 
| \widehat{h}_{t_1-t_2}(mN) |  \ll m^{-100} N^{-100}. 
\end{equation}

Finally, if $m>T^{1+\epsilon}/N$, then since $W$ is contained in an interval of length $T$, we have by Lemma \ref{lmm:Fourier} and taking $j=\lceil 200/\epsilon\rceil+100$

\begin{equation}
| \widehat{h}_{t_1-t_2}(mN) |   \ll_\epsilon \frac{T^{100}}{(mN)^{100}} \Bigl(\frac{T}{T^{1+\epsilon}}\Bigr)^{200/\epsilon}\ll T^{-100}m^{-100} . 
\label{eq:FourierTail}
\end{equation}

With this in mind, we divide the sum into pieces

\begin{equation}
\sum_{\substack{m \in \ZZ^3\setminus\{0\}}} I_m = S_1 + S_2 + S_3, 
\label{eq:SiExpansion}
\end{equation}

\noindent where $S_1$ contains the terms where exactly one $m_i$ is non-zero, $S_2$ contains the terms where exactly two $m_i$ are non-zero, and $S_3$ contains the terms where all three $m_i$ are non-zero.

We will see that $S_1$ is negligible.  In the next section, we will bound $S_2$ using Heath-Brown's theorem,  Theorem \ref{thrm:HeathBrown}.   The main part of the paper is concerned with studying $S_3$, which contains most of the terms and is most difficult.  

%
%

\begin{prpstn}[$S_1$ bound]\label{prpstn:S1}
We have
\[
S_1=O_\epsilon(T^{-10}).
\]
\end{prpstn}

\begin{proof}
  By symmetry, we see that
  
\[
S_1\le 3N^3\sum_{t_1,t_2,t_3\in W}\sum_{\substack{m_3 \not= 0}} |\widehat{h}_{t_1 - t_2}(0) \widehat{h}_{t_2 - t_3}(0) \widehat{h}_{t_3 - t_1}(m_3 N)|.
 \]
 
 By \eqref{eq:FourierTail} (using the trivial bound $|\widehat{h}_t(\xi)|\ll 1$ for the other factors), terms with $|m_3|>T^{1+\epsilon}/N$ contribute
 
 \[
 \ll_\epsilon N^3|W|^3\sum_{m>T^{1+\epsilon}/N}T^{-100}m^{-100}\ll  T^{-10}.
 \]
 
 Thus we may restrict attention to terms with $|m_3|<T^{1+\epsilon}/N$. Next we consider terms with $t_1\ne t_2$. Using \eqref{tdiffm0} to bound $|\widehat{h}_{t_1 - t_2}(0)|$ (and the trivial bound $\widehat{h}_t\ll 1$ for the remaining  factors), we see the terms with $t_1\ne t_2$ and $|m_3|<T^{1+\epsilon}/N$ contribute
 
 \[
 \ll_\epsilon N^3|W|^3\frac{T^{1+\epsilon}}{N}T^{-100}\ll  T^{-10}.
 \]
 
 Similarly, the terms with $t_2\ne t_3$ contribute $O_\epsilon(T^{-10})$. The remaining terms have $t_1 = t_2 = t_3$. For these terms we apply \eqref{tsamemn0} to bound $|\widehat{h}_{t_3 - t_1}(m_3 N)|$, which shows that the terms with $t_1=t_2=t_3$ also contribute $O_\epsilon(T^{-10})$. This gives the result.
\end{proof}

%
%
%
%

\section{The contribution of \texorpdfstring{$S_2$}{S2}}

The aim of this section is to establish the following bound for the sum $S_2$, which ultimately relies on Heath-Brown's estimate, Theorem \ref{thrm:HeathBrown}.

%
%

\begin{prpstn}[$S_2$ bound]\label{prpstn:S2}
For any choice of $k\in \mathbb{N}$
\[
S_{2}\lessapprox_{\epsilon,k} N^2 |W|^2+TN|W|^{2-1/k}+N^2|W|^2 \Bigl(\frac{T^{1/2}}{|W|^{3/4}}\Bigr)^{1/k}.
\]
\end{prpstn}

The proof of this proposition relies on the following consequence of stationary phase, which is part of the well-known  `reflection principle' for Dirichlet polynomials, or the approximate functional equation (values of a Dirichlet polynomial of length $N$ at $t\in [T,2T]$ are determined by values of a Dirichlet polynomial of length $T/N$).

%
%

\begin{lmm}[Approximate functional equation] \label{lmmappfunceq} For every $t$ with $|t| \sim T_0\ge T^\epsilon$, we have
\[
 \Bigl| \sum_{m \not= 0} \widehat{h}_t(mN) \Bigr| \ll \frac{1}{T_0^{1/2}} \int_{u\lessapprox 1}\Bigl| \sum_{m \lessapprox T_0/N} m^{-i(t+u)} \Bigr|du + O_\epsilon( T^{-100} ). 
 \]
 \end{lmm}

Although somewhat standard, we will give a detailed proof of Lemma \ref{lmmappfunceq} below.   Let us first use it to bound $S_2$.  

\begin{proof}[Proof of Proposition \ref{prpstn:S2} assuming Lemma \ref{lmmappfunceq}]
Recall that $S_2$ is the sum of those $I_m$ where exactly two $m_i$ are non-zero.   By symmetry,  we have

\[
S_2 = 3 N^3 \sum_{\substack{m_1, m_2 \not= 0}} \sum_{t_1,t_2, t_3\in W}\widehat{h}_{t_1-t_2}(m_1N)\widehat{h}_{t_2-t_3}(m_2N) \widehat{h}_{t_3 -t_1} (0).
\]

If $t_1 \not= t_3$,  then \eqref{tdiffm0} shows that the last factor $\widehat{h}_{t_3-t_1}(0)$ is $O_\epsilon(T^{-100})$, and so using the bound $\widehat{h}_t(u)\ll (1+|t|^2)/|u|^2$ from Lemma \ref{lmm:Fourier} for the remaining factors, these terms contribute $O_\epsilon(T^{-10})$ in total.   Therefore we have

\[
 S_2 =  3N^3\widehat{h}_{0} (0)\sum_{m_1,m_2\ne 0}  \sum_{t_1,t_2\in W}\widehat{h}_{t_1-t_2}(m_1N)\widehat{h}_{t_2-t_1}(m_2N) +O_\epsilon(T^{-10}).
 \]
 
Since $h_t(u) = w(u)^2 u^{it}$,  we have $h_{-t}(u) = \overline{h_t(u)}$,  and so $ \widehat{h}_{-t} (\xi) = \overline{ \widehat{h}_t( - \xi) }$. In particular,

\[
 \widehat{h}_{t_2-t_1} (m_2N) = \overline{ \widehat{h}_{t_1-t_2}( - m_2N) }. 
 \]
 
Therefore, we can simplify the last equation to get

\[
 S_2 =3N^3 \widehat{h}_0(0)\sum_{t_1, t_2 \in W} \Big| \sum_{m\ne 0} \widehat{h}_{t_1-t_2}(m N) \Big|^2+O_\epsilon(T^{-10}).
  \]

\begin{rmk}  Heath-Brown's theorem, Theorem \ref{thrm:HeathBrown}, 
gives a good estimate for the sum $\sum_{t_1,  t_2 \in W} | \sum_{m \in \ZZ} \widehat{h}_{t_1 - t_2}(mN) |^2$.   However we cannot apply it immediately because we need to be careful to leave out the term with $m=0$.   This term is related to $I_0$ which we handled carefully in the previous section.  
\end{rmk}

If $t_1 = t_2$,  then $\sum_{m \not= 0} \widehat{h}_{t_1 - t_2}(mN)$ is negligible by \eqref{tsamemn0} and \eqref{eq:FourierTail}.   So, splitting the sum dyadically according to the size of $t_1-t_2$, we find

\begin{align*}
 S_2& \lessapprox N^3 \sup_{\substack{M=2^j\\ T^\epsilon/N<M<2T/N}}\sum_{\substack{t_1 \not= t_2 \in W\\ |t_1-t_2|\sim MN}} \Big| \sum_{m \not= 0} \widehat{h}_{t_1-t_2}(m N) \Big|^2 +O_\epsilon(T^{-10}).
 \end{align*}

If $|t_1 - t_2| \sim M N$,  then $\sum_{m \not= 0} \widehat{h}_{t_1 - t_2} (mN)$ can be approximated by a Dirichlet polynomial of length $M$.  Indeed, by Lemma \ref{lmmappfunceq}, for such $t_1,t_2$ we have

\[
 \Bigl| \sum_{m \not= 0} \widehat{h}_{t_1-t_2}(mN) \Bigr| \lessapprox \frac{1}{M^{1/2} N^{1/2}}\int_{|u|\lessapprox 1} \Bigl| \sum_{ m \lessapprox M}  m^{-i(t_1-t_2-u)} \Bigr|du + O_\epsilon( T^{-100} ). 
  \]
  
Squaring and summing over $t_1,t_2\in W$ with $|t_1-t_2|\sim NM$ gives

\[
 S_2 \lessapprox   \sup_{\substack{M\le 2T/N\\ |u|\lessapprox 1}}\frac{N^2}{M}\sum_{\substack{t_1 \not= t_2 \in W\\ |t_1-t_2|\sim MN}} \Big| \sum_{1\le m\lessapprox M}  m^{i(t_1-t_2-u)} \Big|^2 + O_\epsilon( T^{-10} ).
 \]

We can now drop the condition $|t_1-t_2|\sim MN$ for an upper bound, and split the summation range $m\lessapprox M$ into dyadic intervals. We note that the dyadic range when $M=1$ gives a contribution which dominates the error term, so for notational convenience we can absorb the error term into the main sum. Thus we find

\begin{equation}
S_2 \lessapprox_\epsilon \frac{N^2}{M}\sum_{t_1,t_2\in W}\Bigl|\sum_{m \sim M } a_m m^{i(t_1-t_2)}\Bigr|^2
\label{eq:S2Intermediate}
\end{equation}

for some choice of $M\lessapprox T/N$ and some coefficients $|a_m|\le 1$.

We apply H\"older's inequality to this sum, and rewrite the $2k^{th}$ power of the Dirichlet polynomial as the $2^{nd}$ power of a longer Dirichlet polynomial. For any choice of positive integer $k$, we find that

\begin{align}
\sum_{t_1,t_2\in W}\Bigl|\sum_{m \sim M} a_m m^{i(t_1-t_2)}\Bigr|^2&\le |W|^{2-2/k}\Bigl(\sum_{t_1,t_2\in W}\Bigl|\sum_{m\asymp M^k}b_m m^{i(t_1-t_2)}\Bigr|^{2}\Bigr)^{1/k}
\label{eq:S2Holder}
\end{align}

for some coefficients $b_m\le  M^{o_k(1)}$ (by the divisor bound).   Theorem \ref{thrm:HeathBrown} bounds sums of this type.   We recall the statement.

\begin{thrm*}[Heath-Brown]
Let $\mathcal{T}$ be a 1-separated set of reals, contained in an interval of length $T$. Let $|a_n|\lessapprox 1$ be a complex sequence. Then
\[
\sum_{t_1,t_2\in \mathcal{T}}\Bigl|\sum_{n\sim N} a_n n^{i(t_1-t_2)}\Bigr|^2\lessapprox |\mathcal{T}|^2N+|\mathcal{T}|N^2+|\mathcal{T}|^{5/4}T^{1/2}N.
\]
\end{thrm*}

This result implies that

\begin{equation}
\sum_{t_1,t_2\in W}\Bigl|\sum_{m\asymp M^k}b_m m^{i(t_1-t_2)}\Bigr|^{2}\lessapprox_k |W|^2M^k+|W|M^{2k}+|W|^{5/4}T^{1/2}M^k.
\label{eq:S2HeathBrown}
\end{equation}

Substituting \eqref{eq:S2Holder} and \eqref{eq:S2HeathBrown} back into \eqref{eq:S2Intermediate}, we see that 

\begin{align}
S_{2}&\lessapprox_{\epsilon,k}\frac{N^2}{M}(|W|^{2}M+M^2|W|^{2-1/k}+|W|^2 M T^{1/2k}|W|^{-3/4k})\nonumber\\
&\lessapprox_{\epsilon,k} N^2 |W|^2+T N |W|^{2-1/k}+N^2|W|^2 \Bigl(\frac{T^{1/2}}{|W|^{3/4}}\Bigr)^{1/k}.
\label{eq:S2Bound}
\end{align}

This gives the result.
\end{proof}

Now we return to the proof of Lemma \ref{lmmappfunceq}, which roughly says $\widehat{h}_t(mN)$ can be thought of as a smoothed version of $t^{-1/2}m^{it}$ supported on $m\asymp t/N$.

\begin{proof}[Proof of Lemma \ref{lmmappfunceq}]
Since $t\sim T_0\ge T^\epsilon$, by Lemma \ref{lmm:Fourier} we have that $\widehat{h}_t(mN)\ll_\epsilon T^{-100}m^{-2}$ unless $|m|\lessapprox T_0/N$. Thus it suffices to just consider terms with $|m|\le M$ for some suitable $M=T_0^{1+o(1)}/N$ at the cost of an $O_\epsilon(T^{-100})$ error term. We focus on the terms with positive $m$; the terms with negative $m$ can be bounded analogously.

We expand the definition of $\widehat{h}_t$ and truncate the integral using the support of $w$ 

 \begin{align*}
\widehat{h}_t(mN)=\int_{-\infty}^\infty w(u)^2 u^{it}e(- mN u) du
=\int_{1/m}^{2M/m} w(u)^2 u^{it}e(-mN u) du.
\end{align*}

Let $H(s):=\int_{0}^\infty w(u)^2 u^{s-1}du$ be the Mellin transform of $h_0$, which is entire and satisfies $H(s)\ll_j |s|^{-j}$ for any $j\in \mathbb{Z}_{>0}$  when $|\Re(s)|\le 10$ (by repeated integration by parts). Applying Mellin inversion  ($w(u)^2=(2\pi i)^{-1}\int_{1-i\infty}^{1+i\infty} H(s)u^{-s}ds$), we have that

 \begin{align*}
\widehat{h}_t(mN)=\frac{1}{2\pi i}\int_{1/m}^{2M/m}\int_{1-i\infty}^{1+i\infty}H(s)u^{it-s}e(- mN u) ds du.
\end{align*}

By the rapid decay of $H$ we may truncate the $s$ integral to $|s|\lessapprox 1$ at the cost of a $O(T^{-100})$ error term. We then make a change of variables $s=1+ir$ and $v=Nmu$

 \begin{align*}
\widehat{h}_t(mN)=\frac{1}{2\pi}\int_{|r|\lessapprox 1}H(1+ir)(mN)^{i(t-r)}\Bigl(\int_{N}^{2NM}v^{-1+i(t-r)}e(-v) dv\Bigr) dr.
\end{align*}

Summing over $1\le m\le  M$ and applying the triangle bound gives

 \begin{align*}
\sum_{\substack{m\le  M}}\widehat{h}_t(mN)&=\frac{1}{2\pi}\mathop{\int}_{|r|\lessapprox 1}H(1+ir)N^{i(t-r)}\sum_{\substack{1\le m\le M}}m^{i(t-r)}\int_{N}^{2NM}\frac{v^{i(t-r)}e(-v) }{v}dv dr\\
&\ll \int_{|r|\lessapprox 1}\Bigl|\sum_{1\le m\le M}m^{i(t-r)}\Bigr| \Bigl|\int_{N}^{2NM}v^{-1+i(t-r)}e(- v) dv\Bigr| dr.
\end{align*}

Integration by parts and the Van-der-Corput first and second derivative bounds (see \cite[Chapter 3, Lemmas 1 and 2]{M2}) then show that (for $|r|\ge 2$)

\begin{align*}
\int_V^{2V} v^{-1+ir}e(-v)dv&\ll \frac{1}{|r|}, &&\text{ if }V\le |r|/20 \text{ or }V\ge 20|r|,\\
\int_V^{2V} v^{-1+ir}e(-v)dv&\ll \frac{1}{|r|^{1/2}}, &&\text{ if }|r|/20\le V\le 20|r|.
\end{align*}

Since $|t-r|\asymp T_0$ when $r\lesssim 1$ (and recalling that $T^\epsilon\le T_0$, $NM\lessapprox T$), together these give

\[
\int_{N}^{2NM}v^{-1+i(t-v)}e(- v) dv\ll T_0^{-1/2}.
\]

This gives the result.
\end{proof}

%
%
%
%

\section{The contribution of \texorpdfstring{$S_3$}{S3}: a key cancellation}

Now we begin to study $S_3$, which is the most difficult term.  Recall that 

\[
S_3=\sum_{m_1,m_2,m_3 \not= 0} I_m, 
\]

where

\[
 I_m = N^3 \sum_{t_1,t_2,t_3\in W}\widehat{h}_{t_1-t_2}(m_1N)\widehat{h}_{t_2-t_3}(m_2N)\widehat{h}_{t_3-t_1}(m_3N). 
 \]
 
 By Lemma \ref{lmm:Fourier}, $|\widehat{h}_t(\xi)|\ll_j (1+|t|)^j/|\xi|^j$  for any $j\in\mathbb{Z}_{\ge 0}$,  and so $\widehat{h}_t$ is rapidly decaying when $|\xi|$ is much bigger than $|t|$, and hence $I_m$ is negligible unless $|m| \lessapprox T/N$.  Thus 

\begin{equation}
S_3= \sum_{0<|m_1|,|m_2|,|m_3| \lessapprox T/N} I_m+O(T^{-100}).
\label{eq:S3Inter}
\end{equation}

The first step in our argument is an estimate for $|I_m|$.   We introduce the function $R(v)$ which will play an important role in our analysis of $S_3$:

\begin{equation}
R(v):=\sum_{t\in W}|v|^{it} = \widehat{W}\Bigl(\frac{\log{|v|}}{-2\pi}\Bigr), \label{eq:RDef}
\end{equation}

where $\widehat{W}(\xi):=\sum_{t\in W}e^{-2\pi i t \xi}$ is the Fourier transform of the distribution with a delta function at each point of $W$ . In this paper we will occasionally find it convenient to work with $\widehat{W}$, but will not work with the distribution directly.

%
%

\begin{prpstn}[Cancellation within the $I_m$ integrals] \label{prpstnkeycancel} We have

\[
|I_m| \ll N^3  \mathop{\int}_{\substack{|m_1 v_1 + m_2 v_2 + m_3| \lessapprox \frac{1}{N}\\  v_1 \asymp v_2 \asymp 1}}  \Big| R(v_1)R\Bigl(\frac{v_2}{v_1}\Bigr)R(v_2 ) \Big| dv_1 dv_2 + O(T^{-200}).
\]

Moreover, if $|m_1|\le |m_2|\le |m_3|$, then $|I_m| = O(T^{-200})$ unless $|m_2| \asymp |m_3|$.  
\end{prpstn}

\begin{proof} 
To simplify notation, let $w_1(\mathbf{u}):=w(u_1)^2w(u_2)^2w(u_3)^2$. Expanding the definition of $\widehat{h}_t$ as an integral and swapping the order of summation and integration,  we have

\begin{align}
I_m&=N^3 \sum_{t_1,t_2,t_3\in W}\int_{\mathbb{R}^3}e(-N \mathbf{m}\cdot \mathbf{u})w_1(\mathbf{u})u_1^{i(t_1-t_2)}u_2^{i(t_2-t_3)}u_3^{i(t_3-t_1)}d\mathbf{u}\nonumber\\
&= N^3\int_{\mathbb{R}^3}e(-N \mathbf{m}\cdot \mathbf{u})w_1(\mathbf{u}) R\Bigl(\frac{u_1}{u_3}\Bigr)R\Bigl(\frac{u_2}{u_1}\Bigr)R\Bigl(\frac{u_3}{u_2}\Bigr)d\mathbf{u}.\label{eq:ImExpansion}
\end{align}

In \eqref{eq:ImExpansion},  the $R$ functions depend on $u_1/u_3$, $u_2/u_1$ and $u_3/u_2$.   We therefore rewrite the integral using these variables.  
We define $v_1$ and $v_2$ by

$$ v_1 := \frac{u_1}{u_3}, \qquad v_2 := \frac{u_2}{u_3}. $$

We rewrite the integral $I_m$ in terms of the variables $v_1, v_2, u_3$.  
When we change variables,  the $R$ factors depend on $v_1, v_2$ but not on $u_3$.  

\[ R\Bigl(\frac{u_1}{u_3}\Bigr)R\Bigl(\frac{u_2}{u_1}\Bigr)R\Bigl(\frac{u_3}{u_2}\Bigr) = R ( v_1 ) R \Big( \frac{v_2}{v_1} \Big) R \Big( \frac{1}{v_2} \Big) \]

The exponential factor also works out in a nice way in the new variables:

$$ e^{-N \mathbf{m} \cdot \mathbf{u}} = e^{-N(m_1 u_1 + m_2 u_2 + m_3 u_3)} = e^{-N (m_1 v_1 + m_2 v_2 + m_3) u_3}. $$

A Jacobian computation  shows that

$$ du_1 du_2 du_3 = u_3^2 dv_1 dv_2 du_3. $$

So in the new variables, our integral $I_m$ becomes

\[
 N^3 \int_{\mathbb{R}^3} e(-N (m_1 v_1+ m_2 v_2+m_3)u_3)w_{2}(u_3,v_1,v_2)R(v_1)R\Bigl(\frac{v_2}{v_1}\Bigr)R\Bigl(\frac{1}{v_2}\Bigr) dv_1 dv_2 du_3,
 \]
 
where 

\begin{equation}
w_2(u_3,v_1,v_2):=u_3^2w(u_3)^2w(v_1u_3)^2w(v_2u_3)^2.
\label{eq:w2Def}
\end{equation}

Since the $R$ factors do not involve $u_3$, we rewrite our formula to do the $u_3$ integral first:

\[
N^3 \int_{\mathbb{R}^2} \left(  \int_{\mathbb{R}} e(-N (m_1 v_1+ m_2 v_2+m_3)u_3)w_2(u_3,v_1,v_2) du_3 \right) R(v_1)R\Bigl(\frac{v_2}{v_1}\Bigr)R\Bigl(\frac{1}{v_2}\Bigr) dv_1 dv_2.
\]

A key observation in our proof is that we can analyze the norm of this inner integral very accurately using non-stationary phase. Recalling the definition \eqref{eq:w2Def} of $w_2$, we see that for any $j\in \mathbb{Z}_{\ge 0}$, $w_2(u_3,v_1,v_2)$ has $j^{th}$ derivative with respect to $u_3$ bounded by $O_j(1)$ (since $w$ is supported on $[1,2]$ with $\|w^{(\ell)}\|_\infty \ll_\ell 1$ for all $\ell\in \mathbb{Z}_{\ge 0}$). Thus for any $\eta>0$, the inner integral is $O_\eta(T^{-300})$ unless $|m_1 v_1+m_2 v_2+m_3| \le T^\eta/N$ by repeated integration by parts.   In general,  the inner integral has size $\ll 1$.   
In addition, $w_2(u_3,v_1,v_2)$ vanishes unless $v_1, v_2 \in [1/2,2]$, because $w(u)$ is supported on $u\in [1,2]$.   Therefore, the inner integral vanishes unless $v_1 \in[1/2,2]$ and $v_2 \in[1/2,2]$.  Using these bounds for the inner integral and then using the triangle inequality, we see that since $\eta>0$ was arbitrary

\[
|I_m| \ll N^3  \mathop{\int}_{\substack{|m_1 v_1 + m_2 v_2 + m_3| \lessapprox \frac{1}{N}\\ v_1, v_2 \in[1/2,2]}}  \Big| R(v_1)R\Bigl(\frac{v_2}{v_1}\Bigr)R\Bigl(\frac{1}{v_2}\Bigr) \Big| dv_1 dv_2+O(T^{-200}).
\]

Since $|R(v)| = |R(1/v)|$,  we can replace $R(1/v_2)$ by $R(v_2)$.   (This is not really important, but it makes later computations cleaner.)

Finally,  this integral vanishes unless we can find $v_1 \asymp 1$ and $v_2 \asymp 1$ so that $m_1 v_1 + m_2 v_2 + m_3$ is almost zero.  If $|m_1| \le |m_2| \le |m_3|$, this can only happen if $|m_2| \asymp |m_3|$.  This gives the last claim in the proposition. 
\end{proof}

%
%

\begin{rmk}
The cancellation in the inner integral when $|m_1 v_1 + m_2 v_2 + m_3|$ is not $\lessapprox 1/N$ is one of the key observations in our proof.    This cancellation is specific to Dirichlet polynomials as opposed to more general trigonometric polynomials.  For instance, one may consider a `generalized' Dirichlet polynomial of the form $\tilde{D}(t) = \sum_{n \sim N} b_n e^{i t \phi(n)}$, where the function $\phi(n)$ has smoothness and convexity properties similar to those of $\log n$.  One can follow the argument above, but the inner integral will have the form $\int e( N g_{\mathbf{m}, v_1, v_2}(u_3)) du_3$ for some function $g_{\mathbf{m}, v_1, v_2}(u_3)$.  In general, the function $g_{\mathbf{m}, v_1, v_2}(u_3)$ will not be linear (or monomial) in $u_3$ and so we would not be in the special situation where the whole integral is either oscillating or stationary.  Thus one would expect to have to use stationary phase around stationary points in the integral, and the bounds would not work out as they do here.
\end{rmk}

Because of the last claim in Proposition \ref{prpstnkeycancel},  we can restrict attention to $m$ with $0 < |m_1| \le |m_2| \asymp |m_3|$.   The domain of integration can be rewritten in the form

\[ \left| v_2 - \frac{m_1 v_1 + m_3}{-m_2} \right| \lessapprox \frac{1}{|m_2|N} \asymp \frac{1}{|m_3| N}. \]

So the domain of integration is essentially the $\frac{1}{N |m_3|}$-neighborhood of the curve $v_2 = \frac{m_1 v_1 + m_3}{-m_2}$.  Therefore,  $|I_m|$ is morally bounded by

$$\frac{N^3}{N |m_3|} \int_{v_1 \asymp 1} \Big| R ( v_1 ) R \Big( \frac{ m_1 v_1 + m_3}{-m_2 v_1} \Big) R \Big( \frac{m_1 v_1 + m_3}{-m_2} \Big) \Big| dv_1. $$

We can make this rigorous by using a smoothed version of $R$.  Define a smoothed version of $| R(u) |$ in terms of compactly supported bump functions $\tilde{\psi}_1$, $\tilde{\psi}_2$ and a parameter $M\ge 1$ by

\begin{equation} \label{eq:RtDef}
\tilde{R}=\tilde{R}_{\tilde{\psi}_1,\tilde{\psi}_2,M}(u) :=\Big(  \int NM \tilde{\psi}_1(N M (u - u')) \tilde{\psi}_2(u')|R(u')|^2 du' \Big)^{1/2}.
\end{equation}

The following proposition gives an expansion of $S_3$ in terms of such integrals.

%
%

\begin{prpstn}[Expansion of $S_3$]\label{prpstn:S3Expansion}
There is a choice of $1\le M_1\le M\lessapprox T/N$ and a choice of non-negative bump functions $\tilde{\psi}_1$, $\tilde{\psi_2}$ with $\tilde{\psi}_1(x)$ supported on $|x|\lessapprox 1$ and satisfying $\|\tilde{\psi}_1^{(j)}\|_\infty\lessapprox_j 1$ for all $j\in \mathbb{Z}_{\ge 0}$ and $\tilde{\psi}_2(x)$ supported on $x\asymp 1$ and satisfying $\tilde{\psi}_2^{(j)}\ll_j 1$ for all $j\in\mathbb{Z}_{\ge 0}$ and with $\tilde{\psi}_1(0)=\tilde{\psi}_2(1)=1$, such that

\[
S_3 \lessapprox \frac{N^2}{M} \sum_{\substack{|m_1| \sim M_1\\    |m_2|,  |m_3| \asymp M}}  \tilde{I}_m+O(T^{-100}),
\]

where

\[
\tilde{I}_m:= \int_{v_1 \asymp 1}  \Big| R ( v_1) \tilde{R} \Big( \frac{ m_1 v_1 + m_3}{m_2 v_1} \Big) \tilde{R} \Big( \frac{m_1 v_1 + m_3}{m_2} \Big) \Big| dv_1.  
\]
\end{prpstn}

\begin{proof} 
Recall from \eqref{eq:S3Inter} that $S_3$ is bounded by

\[
S_3\le \sum_{0<|m_1|,|m_2|,|m_3| \lessapprox T/N} |I_m| +O(T^{-100}).
\]

From \eqref{eq:ImExpansion}, we have

\[
I_m= N^3\int_{\mathbb{R}^3}e(-N \mathbf{m}\cdot \mathbf{u})w_1(\mathbf{u}) R\Bigl(\frac{u_1}{u_3}\Bigr)R\Bigl(\frac{u_2}{u_1}\Bigr)R\Bigl(\frac{u_3}{u_2}\Bigr)d\mathbf{u}.
\]

From the definition \eqref{eq:RDef} for $R(v)$, we see that $R(1/v) = \overline R(v)$. Therefore we see that $I_{(m_1,m_2,m_3)}=\overline I_{(m_2,m_1,m_3)}$, and similarly for any other transposition of $(m_1,m_2,m_3)$. Thus $|I_m|$ is invariant under any permutation of $(m_1,m_2,m_3)$, and so we can reduce to the case $|m_1|\le |m_2|\le |m_3|$ at the cost of a factor of 6. By Proposition \ref{prpstnkeycancel}, such terms are negligible unless $|m_2|\asymp |m_3|$. Thus, by choosing dyadic scales to maximize the right hand side, we find that there is an $M_1\le M\lessapprox T/N$ such that

\begin{equation}
S_3\lessapprox  \sum_{\substack{|m_1|\sim M_1 \\ |m_2|\asymp M\\ |m_3|\asymp M}} |I_m| +O(T^{-100}).
\label{eq:S3Expansion}
\end{equation}

By Proposition \ref{prpstnkeycancel},  we have for $m_2\asymp M$

\[
|I_m| \ll N^3  \int_{v_1 \asymp 1}  | R(v_1) |  \Biggl(\int\limits_{\substack{v_2\asymp 1\\ | v_2 - \frac{m_1 v_1 + m_3}{-m_2} | \lessapprox  \frac{1}{M N}}}  \Big| R\Bigl(\frac{v_2}{v_1}\Bigr)R( v_2) \Big| dv_2 \Bigg) dv_1+O(T^{-200}).
\]

Using Cauchy-Schwarz (and a change of variables $v_2\mapsto v_2v_1$ for the first factor),  we bound the inner integral by

\[  
\Biggl( \int\limits_{\substack{v_2\asymp 1\\ M N| v_2 - \frac{m_1 v_1 + m_3}{-m_2v_1} | \lessapprox 1 }}  | R(v_2) |^2  dv_2 \Biggr)^{1/2}  \Biggl( \int\limits_{\substack{v_2\asymp 1\\ MN| v_2 - \frac{m_1 v_1 + m_3}{-m_2} | \lessapprox  1}}  | R( v_2 ) |^2 dv_2 \Biggr)^{1/2}.
\]
We can now choose a smooth bump function $\tilde{\psi}_2(v_2)$ which majorizes both the integration constraints $v_2\asymp 1$ and satisfies the support and derivative conditions of the proposition. Similarly, we can choose a bump function $\tilde{\psi}_1$ such that $\tilde{\psi}_1(MN(\frac{m_1v_1+m_3}{-m_2v_1}-v_2))$ majorizes the integration constraint $M N| v_2 - \frac{m_1 v_1 + m_3}{-m_2v_1} | \lessapprox 1$ in the first factor above and $\tilde{\psi}_1(MN(\frac{m_1v_1+m_3}{-m_2}-v_2))$ majorizes the corresponding constraint in the second factor and $\tilde{\psi}_1$ satisfies the support and derivative constraints of the proposition. Recalling the definition \eqref{eq:RtDef} of $\tilde{R}$, we then see that for this choice of $\tilde{\psi}_1,\tilde{\psi}_2$ and $M$, the product of integrals above is
\[ 
\ll \frac{1}{MN}  \tilde{R} \Big( \frac{m_1 v_1 + m_3}{-m_2 v_1} \Big) \tilde{R} \left( \frac{m_1 v_1 + m_3}{-m_2} \right) . 
\]

Thus we find that

\[
|I_{m_1,m_2,m_3}| \ll \frac{N^2}{M} \tilde{I}_{m_1,-m_2,m_3}+O(T^{-100}).
\]

Finally, since we are summing over $m_2$ with $|m_2| \asymp M$, we can replace $- m_2$ with $m_2$ without changing the overall sum. Substituting this into our expression \eqref{eq:S3Expansion} for $S_3$ above then gives the result.
\end{proof}

%
%
%
%

\section{Basic estimate for the low energy case}

In this section,  we begin to estimate $S_3$ using Proposition \ref{prpstn:S3Expansion}.    Recall that $R(v) = \sum_{t \in W} |v|^{it}$.   The best bound for $|R(v)|$ we can hope for is square root cancellation: $|R(v)| \ll |W|^{1/2}$.   If indeed $|R(v)| \approx |W|^{1/2}$ for all $v \asymp 1$, then we get $S_3 \lessapprox N^2 M^2 |W|^{3/2} \lessapprox T^2 |W|^{3/2}$.   We will see more generally that this bound holds whenever the energy of $W$ is very small. 

\begin{prpstn}[\texorpdfstring{$S_3$}{S3} controlled by energy] \label{prpstnsimpleS3} If $W$ is a $T^\epsilon$-separated set contained in an interval of length $T$,  then

\[ S_3 \lessapprox_\epsilon T^2 |W|^{1/2} E(W)^{1/2}. \]

\end{prpstn}

%
%

\begin{rmk}  If we look at the critical case when $T = N^{5/4}$ and $\sigma = 3/4$, then this estimate (together with our bounds for $S_2$) gives an improvement to the basic orthogonality estimate (\ref{basicorth}) when $E(W)$ is significantly below $|W|^{7/3}$.  In the special case when $E(W) \approx |W|^2$, this Proposition is enough to prove our main theorem, Theorem \ref{thrm:LargeValues}.  On the other hand, when $E(W)$ is very large in terms of $|W|$, then we will get good estimates using Lemma \ref{lmm:BasicEnergy}.  However there is an intermediate range of energy that is not yet covered.  In the next two sections, we will develop a strengthening of this proposition that gives new estimates for a wider range of energies.  
\end{rmk}

We begin with some basic lemmas about the moments of $R$.

\begin{lmm}[$L^2$ bound] \label{RL2} Let $W$ be a $T^\epsilon$-separated set contained in an interval of length $T$. Then
\[ \int_{v \asymp 1} |R(v)|^2 dv \ll_\epsilon |W| .\]
\end{lmm}

%
%

\begin{proof} 
Let $\psi_1(v)$ be a smooth bump function which majorizes the range of integration of the integral in the lemma and is supported on $v\asymp 1$ (so satisfies $\|\psi_1^{(j)}\|_\infty \ll_j 1$ for all $j\in \mathbb{Z}_{\ge 0}$). Then we have

\[ \int_{v \asymp 1} |R(v)|^2 dv \le \int \psi_1(v) | R(v) |^2 dv.\]

We substitute the definition of $R(v)$ from \eqref{eq:RDef}, and let $\psi_2(\tau) :=2\pi e^{-2\pi\tau} \psi_1( e^{-2\pi\tau})$. Then, making a change of variables $v=e^{-2\pi \tau}$ gives

\[ \int \psi_1(v) | R(v) |^2 dv = \int  \psi_2(\tau) |  \widehat{W}(\tau)|^2 d \tau= \sum_{t_1, t_2 \in W} \widehat{\psi}_2(t_1 - t_2). \]

Note that $\psi_2$ is a smooth bump around the origin with $\|\psi_2^{(j)}\|_\infty\ll_j 1$ for all $j\in \mathbb{Z}_{\ge 0}$, so $|\widehat{\psi}_2(\xi)|\ll_j |\xi|^{-j}$ for any $j\in \mathbb{Z}_{\ge 0}$. Since $W$ is $T^\epsilon$-separated, if $t_1\ne t_2$ we have that $\widehat{\psi}_2(t_1-t_2)\ll_\epsilon T^{-100}$, so the terms with $t_1\ne t_2$ are negligible. The terms with $t_1=t_2$ contribute $\ll |W|$ to the sum above. Thus the total sum is $O_\epsilon(|W|)$, as required.
\end{proof}

%
%

\begin{lmm}[$L^4$ bound] \label{RL4}
Let $M$, $\tilde{\psi}_1$, $\tilde{\psi}_2$ be as in Proposition \ref{prpstn:S3Expansion}. Then we have

\[  \int_{v \asymp 1} |\tilde{R}(v)|^4 dv \lessapprox E(W)\qquad \text{ and} \qquad \int_{v \asymp 1} |R(v)|^4 dv \lessapprox E(W).  \]

\end{lmm}

%
%

\begin{proof} 
From the definition \eqref{eq:RtDef} of $\tilde{R}$, recalling that $\tilde{\psi}_1,\tilde{\psi}_2\ll 1$ are supported on $|x|\lessapprox 1$ and $|x|\asymp 1$ respectively, and Cauchy-Schwarz, we have

\begin{align*}
\int_{v\asymp 1}|\tilde{R}(v)|^4dv&\lessapprox N^2M^2  \int\limits_{\substack{v\asymp 1\\ |u-v|\lessapprox 1/NM\\ |u'-v|\lessapprox 1/NM}}   |R(u')|^4 du'   du dv\lessapprox \int_{u'\asymp 1}|R(u')|^4 du'.
\end{align*}

Therefore it suffices to prove the result for $R$. We recall that $R(v)=\widehat{W}(\log{|v|})$, so by a change of variables $v=e^\tau$ we see that it suffices to show

\[
\int_{\tau\ll 1}|\widehat{W}(\tau)|^4d\tau\lessapprox E(W).
\]

Let $\eta>0$ and let $\psi_1$ be a smooth bump supported on $\tau\ll 1$ such that $\psi_1(\tau/T^\eta)$ majorizes the range of integration. Then we see that

\begin{align*}
\int_{\tau\asymp 1}|\widehat{W}(\tau)|^4d\tau&\le \int \psi_1\Bigl(\frac{\tau}{T^\eta}\Bigr)|\widehat{W}(\tau)|^4d\tau\\
&=\sum_{t_1,t_2,t_3,t_4\in W}\int \psi_1\Bigl(\frac{\tau}{T^\eta}\Bigr)e(\tau(t_1+t_2-t_3-t_4))d\tau\\
&=T^\eta \sum_{t_1,t_2,t_3,t_4\in W} \widehat{\psi_1}\Bigl(T^\eta(t_3+t_4-t_1-t_2)\Bigr).
\end{align*}

Since $\widehat{\psi_1}$ decays rapidly, we may restrict the summation to $|t_1+t_2-t_3-t_4|\le 1$ at the cost of an $O_\eta(T^{-100})$ error term. The remaining terms contribute $\ll T^\eta E(W)$. Thus, letting $\eta\rightarrow 0$ we obtain

\[
\int_{\tau\asymp 1}|\widehat{W}(\tau)|^4d\tau\lessapprox E(W).\qedhere
\]

\end{proof}

Proposition \ref{prpstnsimpleS3} follows quickly from Lemma \ref{RL2} and Lemma \ref{RL4}, but before we establish this we record a Fourier decay estimate for $\tilde{R}$ which will be needed in later sections. If $f$ is a version of $|R|^2$ smoothed over intervals of length $1/B$,  the Fourier transform of $f$ will decay rapidly beyond $B$, and since $R$ is essentially constant on intervals of length $1/T$, $f$ cannot be too small on a $1/T$-neighbourhood of 1.

\begin{lmm}[Fourier decay of smoothenings of $R$]\label{lmm:SmoothedFourierDecay}
Let $\psi_1,\psi_2,\psi_3$ be smooth non-negative bump functions satisfying:
\begin{enumerate}
\item $\psi_1(t)$ and $\psi_2(t)$ are supported on $t\asymp 1$, and $\psi_3(t)$ is supported on $|t|\lessapprox 1$.
\item $\psi_1(1)=\psi_2(1)=\psi_3(0)=1$.
\item For any $j\in \mathbb{Z}_{\ge 0}$ we have $\psi_1^{(j)},\psi_2^{(j)}\ll_j 1$, $\psi_3^{(j)}\lessapprox_j 1$.
\end{enumerate}
Let $T^\epsilon\le B\lessapprox T$ and
\[
f(u) := \psi_1(u)  \int B \psi_3(B (u - u')) \psi_2(u')|R(u')|^2 du'.
\]

Then for all $j\in\mathbb{Z}_{\ge 0}$ we have
\[\widehat{f}(\xi)\lessapprox_j \frac{T^j}{|\xi|^j}\sup_uf(u).\]
\end{lmm}
\begin{proof}
Let 

\[
f_1(u):=\int B \psi_3(B (u - u')) \psi_2(u')|R(u')|^2 du',\qquad g(u):=\psi_2(u)|R(u)|^2.
\]

 Then $f_1(u)$ is a convolution of $B \psi_3(B u)$ and $g(u)$, so $\widehat{f}_1(\xi)$ has Fourier transform $\widehat{\psi_3}(\xi/B)\widehat{g}(\xi)$. Since $|R(u)|^2\le|W|^2$ and $\psi_2$ is supported on $|x|\asymp 1$, we see that $\widehat{g}(\xi)\ll |W|^2$. The derivative bounds on $\psi_3$ imply that $\widehat{\psi}_3(\xi)\lessapprox_j B^j/(1+|\xi|)^j$ for all $j\in \mathbb{Z}_{\ge 0}$. Thus $\widehat{f}_1(\xi) \lessapprox_j|W|^2 B^j/(1+|\xi|)^j$ for $j\in\mathbb{Z}_{\ge 0}$. Since $f(u)=\psi_1(u)f_1(u)$ we have $\widehat{f}=\widehat{\psi_1}\ast\widehat{f_1}$ and $\widehat{\psi_1}(\xi)\ll_j (1+|\xi|)^{-j}$ for $j\in\mathbb{Z}_{\ge 0}$ from the derivative bound on $\psi_1$. This gives
\[
\widehat{f}(\xi)\lessapprox_j |W|^2\Bigl(\frac{B}{|\xi|}\Bigr)^j\qquad\text{ for all $j\in \mathbb{Z}_{\ge 0}$.}
\]
 Since $W\subseteq [0,T]$, we have that $|R(u)|^2\gg |W|^2$ if $u$ is a sufficiently small multiple of $1/T$ from $1$. Since $\psi_2(1)=\psi_3(0)=1$ we also have that $\psi_2(u)\gg 1$ if $u$ is sufficiently close to 1 and $\psi_3(B(1-u))\gg 1$ if $u$ is a sufficiently small multiple of $1/B$ from 1. Since $B\lessapprox T$, we find from restricting $u$ to a neighbourhood of 1 of width a small multiple of $\min(1/T,1/B)$, that
 \[
f(1)= \int B \psi_3(B(1-u))\psi_2(u) |R(u)|^2 du\gtrapprox  \frac{B}{T}|W|^2.
 \]
 Thus for any $j\in \mathbb{Z}_{\ge 0}$,

\[
\widehat{f}(\xi)\lessapprox_j  |W|^2\Bigl(\frac{B}{|\xi|}\Bigr)^j\lessapprox \frac{TB^{j-1}}{|\xi|^j}f(1)\lessapprox_j \frac{T^j}{|\xi|^j}\sup_uf(u).\qedhere
\]
\end{proof}

We now return to  prove Proposition \ref{prpstnsimpleS3}.

\begin{proof}[Proof of Proposition \ref{prpstnsimpleS3}] Starting with Proposition \ref{prpstn:S3Expansion}, we have for some $M_1\le M\lessapprox T/N$ and some suitable bumps $\tilde{\psi}_1$, $\tilde{\psi}_2$

\[S_3 \lessapprox  \sum_{\substack{|m_1| \sim M_1\\  |m_2|,  |m_3| \asymp M}}  \frac{N^2}{M} \int_{v_1 \asymp 1}  \Big| R ( v_1 ) \tilde{R} \Big( \frac{ m_1 v_1 + m_3}{m_2 v_1} \Big) \tilde{R} \Big( \frac{m_1 v_1 + m_3}{m_2} \Big) \Big| dv_1 . \]

Using H\"older's inequality,  we find that the integral over $v_1$ is bounded by

\begin{align*} 
 \left( \int_{v_1 \asymp 1} |R(v_1)|^2 dv_1 \right)^{1/2}&\left( \int_{v_1 \asymp 1} \left|  \tilde{R} \Big( \frac{ m_1 v_1 + m_3}{m_2 v_1} \Big) \right|^4 dv_1  \right)^{1/4}\\
&\qquad \times  \left( \int_{v_1 \asymp 1} \left|  \tilde{R}  \Big( \frac{m_1 v_1 + m_3}{m_2} \Big) \right|^4 dv_1   \right)^{1/4}. 
\end{align*}

In the second integral we change of variables to $u = \frac{m_1 v_1 + m_3}{m_2 v_1}$ with Jacobian factor $\asymp 1$.   In the third integral change of variables to $u = \frac{m_1 v_1 + m_3}{m_2}$ with a Jacobian factor of norm $\asymp  M / M_1$. Since $\tilde{R}$ is supported on $u\asymp 1$ (from the support of $\tilde{\psi}_2$), we obtain

\[S_3 \ll N^2 M^2  \left( \int_{v_1 \asymp 1} |R(v_1)|^2 dv_1 \right)^{1/2} \left( \int_{u \asymp 1} |\tilde{R} (u) |^4 du  \right)^{1/2}  \]

Using $M \lessapprox T/N$ and Lemmas \ref{RL2} and \ref{RL4},  we find

\[ S_3 \lessapprox T^2 |W|^{1/2} E(W)^{1/2} .\qedhere \]

\end{proof}

%
%

When $E(W)\approx |W|^2$,  the bound from Proposition \ref{prpstnsimpleS3} is the best bound for $S_3$ we know how to prove and corresponds to square-root cancellation in the $R$ function.  For larger $E(W)$, however, we can improve the bound for $S_3$.   Let us indicate the general direction here,  and then we will develop the tool we need in the next section.

Ignoring some technical smoothing,  we morally have

\[
S_3 \lessapprox  \sum_{\substack{|m_1| \sim M_1\\  |m_2|,  |m_3| \asymp M}}  \frac{N^2}{M} \int_{v_1 \asymp 1}  \Big| R ( v_1 ) R \Big( \frac{ m_1 v_1 + m_3}{m_2 v_1} \Big) R \Big( \frac{m_1 v_1 + m_3}{m_2} \Big) \Big| dv_1.  
 \]

We can split up the sum over $\mathbf{m}$ and integral over $v_1$ into pieces where the first $R$ factor has size $\sim A_1$,  the second $R$ factor has size $\sim A_2$, and the third $R$ factor has size $\sim A_3$.   For simplicity,  suppose that $A_1 = A_2 = A_3 = A$, which we expect to be the critical case, and focus on the value of $A$ that dominates the integral.  If $A \approx |W|^{1/2}$,  then we get the bound corresponding to minimal energy.    If $A$ is larger,  then $|R(v)| \sim A$ for only a small subset $U_A \subset \{ v \asymp 1 \}$ by Lemma \ref{RL2}. Our splitting of the summation and integration would mean we have the conditions $v_1\in U_A$, $(m_1v_1+m_3)/(m_2v_1)\in U_A$ and $(m_1v_1+m_3)/m_2\in U_A$ and typically we would think it to be rare for these three sparse conditions to simultaneously hold.  If the simple analysis in Proposition \ref{prpstnsimpleS3} was sharp,  it would mean that for many $v \in U_A$ and many $m_1, m_2, m_3$,  we have $\frac{m_1 v + m_3}{m_2} \in U_A$.   We will see that a small set $U$ cannot be approximately invariant under this large set of affine transformations.   In the next section, we will prove a precise estimate in this spirit, and then we will use it to give stronger bounds for $S_3$ when the energy is greater than $|W|^2$.

%
%
%
%

\section{Summing over affine transformations}\label{sec:Affine}

Given $M>0$ and a compactly supported smooth function $f$, we define

\[
J(f):= \sup_{0<M_1,M_2,M_3<M}\int \Big( \sum_{|m_1| \sim M_1,  m_2 \sim M_2 , |m_3| \ll M_3 }  f \Big( \frac{m_1 u + m_3}{m_2} \Big)  \Big)^2 du,
\]

which is an average of sums of affine transformations of $f$. The aim of this section is to establish the following general bound for $J(f)$.

%
%

\begin{prpstn}[Equidistribution over affine transformations] \label{propsumaff} Suppose that $f(u) $ is non-negative and supported on $u \asymp 1$ and that $|\widehat{f}(\xi)|\lessapprox_j (T/|\xi|)^j\sup_u |f(u)|$ for all $j\in \mathbb{Z}_{\ge 0}$.  Then 

\[ J(f) \lessapprox M^6 \Big( \int f(u) du \Big)^2 + M^4 \int f(u)^2 du.\]

\end{prpstn}

%
%

To avoid any possible confusion over the dependencies of the implied constants in the statement above, we emphasize that this is saying that if $f:\mathbb{R}\rightarrow [0,\infty)$ and $T\in \mathbb{R}$ have the property that there is a constant $C>0$ such that $f$ is supported on $[1/C,C]$ and the property that for every $\epsilon>0$ and $j\in\mathbb{Z}_{\ge 0}$ there is a constant $c(j,\epsilon)$ such that $|\widehat{f}(\xi)|\le c(j,\epsilon)T^\epsilon (T/|\xi|)^j\|f\|_\infty$ (for all $\xi \in \mathbb{R}$), then we can conclude that for every $\delta>0$ there is a constant $c'(\delta)>0$, depending only on $C$ and all the constants $c(j,\epsilon)$, such that

\[
J(f)\le c'(\delta)T^\delta(M^6\|f\|_1^2+M^4\|f\|_2^2).
\]

We emphasize that in this section all implied constants from our $\ll$ and $\lessapprox$ notation may depend on the implied constants $C$ and $c(j,\epsilon)$ from the assumptions on $f$.

%
%

\hfill\\

\begin{rmk}
By a simple Cauchy-Schwarz argument, we can bound the left hand side by $M^6 \int |f(u)|^2 du$.  This bound is tight if $f(u)$ is a smooth bump on $u \sim 1$.  The Proposition improves on this bound when $f(u)$ is sparse.  To get a sense of what it means, it is good to imagine the example that $f = \chi_U$ where $U$ is a union of (smoothed) $1/T$-intervals.  The proposition says that if $|U|$ is much smaller than $1$, then the sets $\{ \frac{m_1 U + m_3}{m_2} \}_{m_1, m_2, m_3 \sim M}$ cannot overlap too much.  

There are two key examples which correspond to the two terms on the right-hand side.  If $U$ is a random subset, then the left-hand side is comparable to $M^6 \left( \int f(u) du \right)^2$.  If $U$ is the $1/T$-neighborhood of the set of rational numbers $r/s$ with $r, s \sim B$ (for some $B<T^{1/2}$), then the left-hand side is comparable to $M^4 \int |f(u)|^2 du$. Indeed, if $B>M$ and $u=r/s+O(1/T)$ where $r,s\sim B$ and $s=d e$ with $d\sim M$, then we see from restricting to the terms with $m_1=d$ that
\[
\sum_{m_1,m_2,m_3\sim M}f\Bigl(\frac{m_1u+m_2}{m_3}\Bigr)\ge \sum_{m_2,m_3\sim M}f\Bigl(\frac{r+em_2}{em_3}+O(|u-r/s|)\Bigr)\gg M^2
\]
whenever $|u-r/s|\le c/T$ for a sufficiently small constant $c$, since $(r+em_2)/(em_3)$ is also a rational with numerator and denominator of size $B$. Thus the integrand of $J(f)$ is of size $M^4$ on a set of measure $\gtrapprox B^2/T\asymp \|f\|^2_2$. Similarly, if $B\le M$ and $u=r/s+O(1/T)$ we can restrict the summand to $s|m_3$ and $\gcd(m_1,rm_3/s+m_2)\sim M/B$ and achieve the same bound.
\end{rmk}

%
%

The following lemma is the main technical result used to prove Proposition \ref{propsumaff}, which is based on a fairly long Fourier analytic argument.

\begin{lmm}[Iterative bound for $J(f)$] \label{lmm:JIteration}Let $f$ be as in Proposition \ref{propsumaff}. Then there is a bump function $\psi(x)$ supported on $|x|\lessapprox 1$ such that

\[J(f)\lessapprox M^6 \left( \int f(u) du \right)^2 + \Bigl(M^4 \int f(u)^2 du\Bigr)^{1/2}J(\tilde{f})^{1/2},\]

where $\tilde{f}$ is defined in terms of $\psi(x)$ by

\[
\tilde{f}(u) := T \int  \psi( T(u - u') ) f(u') du'.
\]

\end{lmm}

%
%

As above, the bump function $\psi(x)$ may depend on the implied constants in the assumptions on $f$. If $f$ satisfies $\widehat{f}(\xi)\le c(j,\epsilon)T^\epsilon (T/|\xi|)^j\sup_u f(u)$ and is supported on $[1/C,C]$, then Lemma \ref{lmm:JIteration} claims that there is a non-negative 1-bounded function $\psi(x)$ such that for all $\delta>0$ and $j\in \mathbb{Z}_{\ge 0}$ there are constants $c'(\delta)$ and $c'(j,\delta)$ (depending only on $C$ and the constants $c(j,\epsilon)$) such that $\psi(x)$ is supported on $|x|\le c'(\delta)T^\delta$, $|\psi^{(j)}(\xi)|\le c'(j,\delta)T^\delta$ and such that

\[
J(f)\le c'(\delta)T^\delta (M^6 \|f\|_1^2+M^2\|f\|_2J(\tilde{f})^{1/2}).
\]

Before we prove Lemma \ref{lmm:JIteration}, we first show how to deduce Proposition \ref{propsumaff} from it.

\begin{proof}[Proof of Proposition \ref{propsumaff} assuming Lemma \ref{lmm:JIteration}]
We wish to show that for any $\epsilon>0$ there is a $C(\epsilon)>0$ such that

\begin{equation}
J(f)\le C(\epsilon)T^\epsilon \Bigl( M^6 \Big( \int f(u) du \Big)^2 + M^4 \int f(u)^2 du\Bigr),
\label{eq:JTarget}
\end{equation}

for any function $f$ satisfying the assumptions of Proposition \ref{propsumaff} (with the constant $C(\epsilon)$ depending only on $\epsilon$ and the implied constants of the assumptions on $f$). 

We wish to prove this by downwards induction on $\epsilon$. As a base case, the result clearly holds for $\epsilon=100$. By induction, we can assume that \eqref{eq:JTarget} holds with $3\epsilon/2$ in place of $\epsilon$ and seek to establish \eqref{eq:JTarget}. We first apply Lemma \ref{lmm:JIteration} to give

\[
J(f)\lessapprox M^6 \Big( \int f(u) du \Big)^2 + \Bigl(M^4 \int f(u)^2 du\Bigr)^{1/2}J(\tilde{f})^{1/2}.
\]

Since $f(u)$ is supported on $u\asymp 1$ and $\psi(x)$ is supported on $x\lessapprox 1$, we see that $\tilde{f}(u)$ is also supported on $u\asymp 1$. Moreover, $\tilde{f}$ has Fourier transform $\widehat{f}(\xi)\widehat{\psi}(\xi/T)$, so also satisfies the rapid decay assumption of Proposition \ref{propsumaff}. Thus $\tilde{f}$ satisfies the assumptions of Proposition \ref{propsumaff} (with implied constants depending only on the implied constants from the assumptions on $f$) and so we may apply the induction hypothesis to $J(\tilde{f})$, which shows

\[
J(\tilde{f})\ll_\epsilon T^{3\epsilon/2}\Bigl( M^6 \Big( \int \tilde{f}(u)  du\Big)^2 + M^4 \int \tilde{f}(u)^2 du\Bigr).
\]

Since $\tilde{f}$ is a smoothed version of $f$, we can bound $\int \tilde{f}(u) du \lessapprox \int f(u) du$ and $\int \tilde{f}(u)^2 du \lessapprox \int f(u)^2 du$. Substituting these bounds back into our bound for $J(f)$, we find

\begin{align*}
 J(f)&\lessapprox_\epsilon T^{3\epsilon/4}\Biggl(M^6 \Big( \int f(u)  du\Big)^2+M^4 \int f(u)^2 du\Biggr).
 \end{align*}
 
Thus \eqref{eq:JTarget} holds for $C(\epsilon)$ sufficiently large, which completes the induction.
\end{proof}

%
%

We now return to the proof of Lemma \ref{lmm:JIteration}.

\begin{proof}[Proof of  Lemma \ref{lmm:JIteration}]
The most interesting situation is when $M_1 = M_2 = M_3 = M$.  We encourage the reader to keep this case in mind on first reading. 

By rescaling we may assume that $\sup_u f(u)\asymp 1$. Crudely, the rapid decay of $\widehat{f}$ implies that $f'(u)\lessapprox T^2$, so $\int f(u)du\gtrapprox 1/T^2$ and
\begin{align*}
\sum_{m_3\in\mathbb{Z} }f\Bigl(\frac{m_1v+m_3}{m_2}\Bigr)&=\sum_{\substack{m_3\\ |m_1v+m_3|\asymp m_2}}\Bigl(\int_{m_3}^{m_3+1}f\Bigl(\frac{m_1v+u}{m_2}\Bigr)du+O\Bigl(\frac{\|f'\|_\infty}{m_2}\Bigr)\Bigr)\\
 &\lessapprox M\int f(u)du+T^2,
\end{align*}
which implies the result if $M>T^4$. Thus we may assume that $M\le T^{4}$. We may assume that all the implied constants in the assumptions on $f$ are sufficiently small compared with $T$ since otherwise the result is trivial. Moreover, since $\int f(u)du\gtrapprox T^{-2}$ and $M\le T^4$, any $O(T^{-100})$ additive factors are negligible. 

We let $\psi_1(x)$ be a smooth bump supported on $|x| \ll 1$ so that $\psi_1(m_3/M_3)$ majorizes the summation condition $|m_3|\ll M_3$.   Thus we can bound the inner sum in $J(f)$ by

\[
 g(u) :=  \sum_{|m_1| \sim M_1, m_2 \sim M_2} \sum_{m_3 }  \psi_1\Bigl(\frac{m_3}{M_3}\Bigr) f \Big( \frac{m_1 u + m_3}{m_2}  \Big). 
 \]

Squaring and integrating over $u$, and then applying Plancherel gives (for the choice of $M_1,M_2,M_3$ achieving the supremum)

\begin{equation}
J(f)\le \int |g(u)|^2 du=\int |\widehat{g}(\xi)|^2 d\xi.
 \label{eq:FourierInt}
\end{equation}

We wish to estimate $\widehat{g}(\xi)$.  We have

\[ \widehat{g}(\xi) =   \sum_{|m_1| \sim M_1, m_2 \sim M_2} \int \sum_{m_3 } \psi_1\Bigl(\frac{m_3}{M_3}\Bigr) f \Big( \frac{m_1 u + m_3}{m_2}  \Big) e(- \xi u) du.  \]

We do a change of variables: $\tilde{u} = u  + \frac{m_3}{m_1}$, so that $f( \frac{m_1 u + m_3}{m_2} ) = f( \frac{m_1 \tilde{u}}{m_2})$.  In the new variables, we get

\[ \widehat{g}(\xi) =   \sum_{|m_1| \sim M_1, m_2 \sim M_2} \left( \int f \Big( \frac{m_1 \tilde{u}}{m_2}  \Big) e(- \xi \tilde{u}) d \tilde{u}\right)  \left( \sum_{m_3 \in \mathbb{Z}}   \psi_1\Bigl(\frac{m_3}{M_3}\Bigr)  e\Bigl(\frac{m_3}{m_1} \xi\Bigr) \right).  \]

The  integral in parentheses is $\frac{m_2}{m_1} \widehat{f}( \frac{m_2}{m_1} \xi )$.   The first key point in our analysis is that we can explicitly do the last sum by  Poisson summation.  It is equal to $M_3 \sum_{\ell} \widehat{\psi}_1( M_3 (\ell - \frac{\xi}{m_1}))$.    So all together we have

\[
 \widehat{g}(\xi) = \sum_{|m_1| \sim M_1} M_3 \sum_{\ell } \widehat{\psi}_1 \Bigl(M_3 \Bigl(\ell - \frac{\xi}{m_1}  \Bigr) \Bigr)  \sum_{m_2 \sim M_2}  \frac{m_2}{m_1} \widehat{f} \Big( \frac{m_2}{m_1} \xi \Big). 
 \]

Since $\widehat{\psi}_1$ is rapidly decaying, $\widehat{\psi}_1(M_3(\ell - \frac{\xi}{m_1} ))$ is negligible unless $| \xi - \ell m_1 | \lessapprox \frac{M_1}{M_3}$.  Therefore, we have

\begin{equation} 
| \widehat{g}(\xi) | \le \sum_{|m_1| \sim M_1}\sum_{\ell: |\xi - m_1 \ell| \lessapprox \frac{M_1}{M_3}} M_3  \Big|   \sum_{m_2 \sim M_2} \frac{m_2}{m_1} \widehat{f} \Big( \frac{m_2}{m_1} \xi \Big) \Big| +O(T^{-100}).
\label{hatg} 
\end{equation}

We have to estimate $\int_\RR |\widehat{g}(\xi)|^2 d \xi$.  
For a small $\eta>0$, we break up the domain of integration into the region $|\xi| \le T^\eta M_1/M_3$, the region $T^\eta M_1/M_3<|\xi|\le T^6$ and the remainder: 

\begin{equation}
 \int_\RR |\widehat{g}(\xi)|^2 d \xi =  \underbrace{ \int_{|\xi| \le T^\eta \frac{M_1}{M_3}} |\widehat{g}(\xi)|^2 d \xi}_{I} + \underbrace{\int_{T^\eta \frac{M_1}{M_3}<|\xi| \le T^6} |\widehat{g}(\xi)|^2 d \xi}_{II} + \underbrace{\int_{ T^6<|\xi|} |\widehat{g}(\xi)|^2 d \xi}_{III}. \label{eq:IntegralSplit}
 \end{equation}

If $|\xi|>T^6$, then since $\widehat{f}(\xi)$ is rapidly decaying for $|\xi| > T$ and $M\le T^4$, we see that $\widehat{f}(m_2\xi/m_1)$ is negligible and $|\widehat{g}(\xi)|\ll T^{-100}|\xi|^{-2}$. Thus

\begin{equation}
III=O(T^{-100}).
\label{eq:IIIBound}
\end{equation}

 If $|\xi| \le T^\eta M_1/M_3$, then in \eqref{hatg}, the only terms that contribute have $|\ell| \ll T^\eta/M_3\ll T^\eta$.  Thus, by Cauchy-Schwarz we obtain from \eqref{hatg}

\begin{align*}
 | \widehat{g}(\xi) |^2 &\ll T^{2\eta} M_1 \sum_{|m_1| \sim M_1} M_3^2  \Big|   \sum_{m_2 \sim M_2} \frac{m_2}{m_1} \widehat{f} \Big( \frac{m_2}{m_1} \xi \Big) \Big|^2+O(T^{-100}) \\
&\ll T^{2\eta} M_2^4 M_3^2 \sup_{\xi} |\widehat{f}(\xi)|^2.
\end{align*}

We absorbed the $O(T^{-100})$ error term in our bound since $\widehat{f}(0)\gtrapprox T^{-2}$. Therefore 

\begin{align} 
I = \int_{|\xi| \le T^\eta M_1/M_3} |\widehat{g}(\xi)|^2 d \xi &\ll T^{3\eta} M_1 M_2^4 M_3 \sup_{\xi} |\widehat{f}(\xi)|^2\nonumber \\
& \ll T^{3\eta}M^6 \left( \int f(u) du \right)^2. \label{eq:IBound} 
\end{align}

Now suppose that $ T^\eta M_1/M_3<|\xi|\le T^6$. We return to \eqref{hatg} and consider the number of terms in the outer double sum. Let $s=m_1 \ell$. In this range, $s$ must be a non-zero integer in the $\lessapprox M_1/M_3$ neighborhood of $\xi$ as soon as $T$ is sufficiently large in terms of $\eta$.  The number of such integers $s$ is $\lessapprox 1 + M_1/M_3$.  Since $|\xi| \le T^6$ and $s$ is non-zero, each such integer $s$ has $\lessapprox 1$ factorizations as $s=m_1\ell$.  All together the number of terms in the outer double sum is $\lessapprox  1 + M_1/M_3$.  Therefore, we can use Cauchy-Schwarz in \eqref{hatg} to bound $|\widehat{g}(\xi)|^2$ by 

\[ 
\lessapprox_\eta \left(1 + \frac{M_1}{M_3} \right) \sum_{|m_1| \sim M_1}\sum_{\substack{\ell \\ |\xi - \ell m_1| \lessapprox M_1/M_3}}  M_3^2  \Big|   \sum_{m_2 \sim M_2} \frac{m_2}{m_1} \widehat{f} \Big( \frac{m_2}{m_1} \xi \Big) \Big|^2 +O(T^{-200}).
\]

Therefore the term $II$ is bounded by

\[\lessapprox_\eta  (M_1 M_3 + M_3^2) \sum_{|m_1| \sim M_1}\sum_{\ell}   \int\limits_{|\xi - \ell m_1| \lessapprox M_1/M_3}  \Big|   \sum_{m_2 \sim M_2} \frac{m_2}{m_1} \widehat{f} \Big( \frac{m_2}{m_1} \xi \Big) \Big|^2 d\xi+O(T^{-100}). \]

Morally this integral does not depend on $m_1$, and we can make this precise by changing variables. For each $m_1, \ell$, we write $\xi = \ell m_1 + \frac{m_1}{M_3} \tau $ and do a change of variables to get (extending the range of integration slightly for an upper bound so we have a range independent of $m_1$)

\[ II \lessapprox_\eta (M_1 +  M_3) \sum_{\ell }  \int_{|\tau| \lessapprox 1}  \Big|   \sum_{m_2 \sim M_2} m_2 \widehat{f} \Big( \ell m_2 + \frac{m_2}{M_3} \tau \Big) \Big|^2 d \tau+O(T^{-100}). \]

Since $\widehat{f}(\xi)=O(T^{-200})$ unless $|\xi | \lessapprox T$, we can restrict the sum over $\ell$ to the range $| \ell | \lessapprox T/ M_2$ at the cost of a negligible error.  We introduce a bump  $\psi_2(x)$ supported on $|x| \lessapprox 1$ so that $\psi_2(M_2\ell/T)$ majorizes this summation condition, and bound the last expression by

\[ \lessapprox_\eta (M_1 + M_3)\Sigma_{II}+O(T^{-100}), \]

where

\[ \Sigma_{II}:= \sum_{\ell}  \psi_2 \Big(\frac{M_2 \ell}{T} \Big) \int_{| \tau| \lessapprox 1}   \Big|   \sum_{m_2 \sim M_2}  m_2 \widehat{f} \Big( \ell m_2 + \frac{m_2}{M_1}  \tau \Big) \Big|^2 d  \tau.\]

We write out $\widehat{f}$ as an integral, expand out the square and bring the summation over $\ell$ and integration over $\tau$ on the inside to get

\begin{equation}
 \Sigma_{II}=  \int \int \sum_{m_2, m_2' \sim M_2}m_2 m_2' f(u) f(u') Z_1 Z_2 du' du, \label{eq:SigmaII}
 \end{equation}

where $Z_1=Z_1(m_2,m_2',u,u')$ and $Z_2=Z_2(m_2,m_2',u,u')$ are given by

\begin{align*}
Z_1&:=  \int_{| \tau| \lessapprox 1} e\Bigl(\tau\Bigl( \frac{m_2'}{M_1}u' - \frac{m_2}{M_1} u\Bigr)\Bigr)d\tau ,\\
 Z_2&:= \sum_{\ell}  \psi_2 \Big(\frac{M_2 \ell}{T} \Big) e\Bigl(\ell\Bigl(m_2' u' - m_2 u\Bigr)\Bigr).
\end{align*}

Trivially we have $|Z_1|\lessapprox 1$. By Poisson summation, and the rapid decay of $\widehat{\psi}_2$, we have

\begin{align*}
Z_2&=\frac{T}{M_2}\sum_{j}  \widehat{\psi}_2\Bigl(\frac{j-m_2'u'+m_2u}{M_2/T}\Bigr)\\
&=\frac{T}{M_2}\sum_{\substack{j\\ |j-m_2'u'+m_2u|\lessapprox M_2/T}}  \widehat{\psi}_2\Bigl(\frac{j-m_2'u'+m_2u}{M_2/T}\Bigr)+O(T^{-100}).
\end{align*}

Substituting these back into our expression \eqref{eq:SigmaII} for $\Sigma_{II}$ and recalling $M\le T^4$, we find

\begin{align*}
\Sigma_{II}\lessapprox M_2\int f(u) \sum_{m_2, m_2' \sim M_2}\sum_{j}   \int_{|u' - \frac{m_2 u + j}{m_2'} | \lessapprox \frac{1}{T}} T f(u') du' du+O(T^{-80}).
\end{align*}

This gives
\[ II\lessapprox_\eta M_2 (M_1 + M_3)  \int f(u)  \sum_{m_2, m_2' \sim M_2}\sum_{j}  \int_{|u' - \frac{m_2 u + j}{m_2'} | \lessapprox \frac{1}{T}}  T f(u') du' du+O(T^{-70}). \]

Since $\widehat{f}(\xi)$ rapidly decays for $|\xi|>T$, $f$ is morally almost constant on intervals of length $1/T$. We let $\psi(x)$ be a smooth bump function supported on $x\lessapprox 1$, and then define $\tilde{f}$ to be a slightly smoothed version of $f$ given by 

\[
\tilde{f}(u) := \int T \psi( T(u - u') ) f(u') du',
\]

as in the statement of Lemma \ref{lmm:JIteration}

We can choose $\psi$ such that the integral over $u'$ in the bound for $II$ above is bounded by $\tilde{f} (\frac{m_2u+j}{m_2'})$ and $\tilde{f}(u)$ is supported on $u\asymp 1$ (recall that the range of support of $f$ is assumed to be sufficiently small in terms of $T$). Thus we find (recalling that $f$ is supported on $u\asymp 1$ and $M_1,M_2,M_3\le M$)

\[
 II \lessapprox_\eta  M^2 \int_{u\asymp 1} f(u)  \sum_{m_2, m_2' \sim M_2}\sum_{ j} \tilde{f} \Big( \frac{m_2 u + j}{m_2'} \Big) du+O(T^{-70}).
 \]

Now we apply Cauchy-Schwarz to get

\begin{equation}
II \lessapprox_\eta \Big[ M^4 \int f(u)^2 du \Big]^{1/2} \Big[ \int_{u\asymp 1} \Big( \sum_{m_2, m_2' \sim M_2,  j \in \mathbb{Z}} \tilde{f} \Big( \frac{m_2 u + j}{m_2'} \Big) \Big)^2 du\Big]^{1/2}. \label{eq:IIBound}
\end{equation}

(The $O(T^{-70})$ term it is clearly majorized by the above expression and so may absorbed into the implied constant.) Since $u\asymp 1$ and $f(x)$ is supported on $x\asymp 1$, we may restrict the summation over $j$ to $j\ll M_2$. Thus we see that the second term in square brackets is bounded by $J(\tilde{f})$. Putting together \eqref{eq:FourierInt}, \eqref{eq:IntegralSplit}, \eqref{eq:IIIBound}, \eqref{eq:IBound} and \eqref{eq:IIBound} we find that for any $\eta>0$, provided $T$ is sufficiently large in terms of $\eta$, we have

\begin{align*}
& J(f)\lessapprox_\eta  T^{3\eta}M^6 \left( \int f(u) du \right)^2+ \Big( M^4 \int f(u)^2 du\Big)^{1/2}J(\tilde{f})^{1/2}.
\end{align*}

This now gives the result on letting $\eta\rightarrow 0$ sufficiently slowly with $T$.
\end{proof}

%
%
%
%

\section{Further bounds for \texorpdfstring{$S_3$}{S3}}

In this section,  we use our bounds for sums over affine transformations to improve our bound for $S_3$.   We will get the following estimate.

%
%

\begin{prpstn}[Refined $S_3$ bound] \label{prpstnS3} If $W$ is a $T^\epsilon$-separated set contained in an interval of length $T$,  then

\begin{equation}
S_{3} \lessapprox_\epsilon T^2 |W|^{3/2}+TN|W|^{1/2}E(W)^{1/2}.
\label{eq:S3DirichletBound}
\end{equation}

\end{prpstn}

%
%

\begin{rmk}
Comparing with Proposition \ref{prpstnsimpleS3},  when $E(W) \approx |W|^2$ both propositions give $S_3 \lessapprox T^2 |W|^{3/2}$.  When $E(W)$ is large,  Proposition \ref{prpstnS3} is better by a factor of $N/T$.  If we look at the critical case when $T = N^{5/4}$ and $\sigma = 3/4$, then this estimate (together with our bounds for $S_2$) gives an improvement to the basic orthogonality estimate (\ref{basicorth}) when $E(W)$ is significantly below $|W|^{3}$.  In other words, this Proposition gives an improvement unless the energy is essentially maximal.  When $E(W)$ is very large, however, we will be able to get good estimates in the next section using Theorem \ref{thrm:HeathBrown}. 
\end{rmk}

\begin{proof}

Recall from Proposition \ref{prpstn:S3Expansion} that $S_3\lessapprox S_{3,0}+O(T^{-100})$, where

\[  
S_{3,0}:=\frac{N^2}{M} \int_{v_1 \asymp 1}  | R ( v_1 ) |   \sum_{\substack{|m_1| \sim M_1\\  |m_2|,  |m_3| \asymp M}} \left| \tilde{R} \Big( \frac{ m_1 v_1 + m_3}{m_2 v_1} \Big) \tilde{R} \Big( \frac{m_1 v_1 + m_3}{m_2} \Big)\right|  dv_1   
\]

for some $1\le M_1\le M\lessapprox T/N$ and some suitable smooth bumps $\tilde{\psi}_1$, $\tilde{\psi}_2$ which appear in the definition \eqref{eq:RtDef} of $\tilde{R}=\tilde{R}_{\tilde{\psi}_1,\tilde{\psi}_2,M}$. By Cauchy-Schwarz we have

\[
S_{3,0}\le \frac{N^2}{M} S_{3,1}^{1/2}S_{3,2}^{1/2},
\]

where (using Lemma \ref{RL2} to bound $S_{3,1}$)

\begin{align*}
S_{3,1}&:=\int_{v \asymp 1}  | R(v)|^2dv\ll_\epsilon |W|,\\
S_{3,2}&:=\int_{v \asymp 1}  \Bigl( \sum_{|m_1| \sim M_1,\,   |m_2|,  |m_3| \asymp M}  \Bigr| \tilde{R}\Bigl(\frac{m_3+m_1v}{m_2v}\Bigr) \tilde{R}\Bigl(\frac{m_3+m_1v}{m_2}\Bigr)\Bigr|\Bigr)^2dv.
\end{align*}

By Cauchy-Schwarz again, we have that

\[
S_{3,2}\ll S_{3,3}^{1/2}S_{3,4}^{1/2},
\]

where 

\begin{align*}
S_{3,3}&:=\int_{v \asymp 1}  \Bigl(  \sum_{|m_1| \sim M_1, \,  |m_2|,  |m_3| \asymp M} \Bigr|\tilde{R}\Bigl(\frac{m_3+m_1v}{m_2v}\Bigr)\Bigr|^2\Bigr)^2dv,\\
S_{3,4}&:=\int_{v \asymp 1}  \Bigl(  \sum_{|m_1| \sim M_1, \,  |m_2|,  |m_3| \asymp M} \Bigr|\tilde{R}\Bigl(\frac{m_3+m_1v}{m_2}\Bigr)\Bigr|^2\Bigr)^2dv.
\end{align*}

We bound $S_{3,3}$ and $S_{3,4}$ using Proposition \ref{propsumaff}.  To bound $S_{3,4}$, we use $f(v) =  \psi_1(v) |\tilde{R}(v)|^2$, where $\psi_1(v)$ is a smooth bump supported on $v \asymp 1$ taking a maximal value of 1 at $v=1$ which majorizes the range of integration.   To control $S_{3,3}$, we make a change of variables $u=1/v$ and rewrite $S_{3,3}$ as 

\begin{align*}
S_{3,3}&\ll \int_{u \asymp 1} \Bigl( \sum_{|m_1| \sim M_1, \,  |m_2|,  |m_3| \asymp M} \Bigr|\tilde{R}\Bigl(\frac{m_3u+m_1}{m_2}\Bigr)\Bigr|^2\Bigr)^2du.
\end{align*}

Then we use $f(u) = \psi_1(u) |\tilde{R}(u)|^2$ again.  Lemma \ref{lmm:SmoothedFourierDecay} shows that $\widehat{f}(\xi)\lessapprox_j \| f\|_\infty T^j/|\xi|^j $ (taking `$\psi_2$' to be $\widetilde{\psi}_2$,  `$\psi_3$' to be $\widetilde{\psi}_1$ and `$B$' to be $MN$), and so $f$ satisfies the Fourier decay conditions of Proposition \ref{propsumaff}. Moreover, $f$ is clearly has support on $u\asymp 1$ from the support of $\psi_1$. Thus Proposition \ref{propsumaff} gives the bounds

\begin{align*}
 S_{3,3}, \, S_{3,4}
 &\lessapprox M^6 \Big( \int_{v \asymp 1} |\tilde{R}(v)|^2 dv \Big)^2 + M^4 \int_{v \asymp 1} |\tilde{R}(v)|^4 dv .
 \end{align*}

Applying Lemmas \ref{RL2} and \ref{RL4}, we get

\[ S_{3,3}, \,S_{3,4}  \lessapprox_\epsilon M^6 |W|^2 + M^4 E(W). \]

The same bound holds for $S_{3,2}$ since $S_{3,2} \le S_{3,3}^{1/2} S_{3,4}^{1/2}$.    Then we get

\begin{align*}
 S_{3,0} \lessapprox \frac{N^2}{M} S_{3,1}^{1/2} S_{3,2}^{1/2} &\lessapprox_\epsilon \frac{N^2}{M} |W|^{1/2} (M^6 |W|^2 + M^4 E(W))^{1/2} \\
 &= N^2 M^2 |W|^{3/2} + N^2 M |W|^{1/2} E(W)^{1/2}. 
 \end{align*}

Since $M \lessapprox T/N$ and $S_3\lessapprox S_{3,0}+O(T^{-100})$, we get

\[ S_3 \lessapprox_\epsilon T^2 |W|^{3/2} + T N  |W|^{1/2} E(W)^{1/2}. \qedhere\]

\end{proof}

%
%
%
%

\section{Energy Bound}\label{sec:Energy}

In this section, we prove bounds related to the energy of $W$,  which show that a Dirichlet polynomial cannot be too large on a set of large energy. These bounds ultimately rely on Heath-Brown's bound Theorem \ref{thrm:HeathBrown}, and are closely related to (and refine) arguments in \cite{HB2}. Recall that in equation \eqref{eq:EnergyDef}, we defined the energy of a finite set $W \subset \mathbb{R}$ by

\[
E(W) := \#\{t_1,t_2,t_3,t_4\in W:\,|t_1+t_2-t_3-t_4|\le 1\}.
\]

We will prove two bounds about the behavior of Dirichlet polynomials on sets of high energy.  The first bound is Lemma \ref{lmm:BasicEnergy}.  We recall the statement here.

\begin{lmm*}
Let $N\in [T^{2/3},T]$, $\sigma > 1/2$ and $D(t) = \sum_{n\sim N} b_n n^{i t }$ with $|b_n| \le 1$. Suppose $W \subset [0,T] $ is a 1-separated set such that $|D(t)| > N^\sigma$ for $t\in W$. Then

\[
E(W)\le |W|^3 N^{1-2\sigma+o(1)}+|W|^2 N^{2-2\sigma+o(1)}.
\]

\end{lmm*}

Combining Lemma \ref{lmm:BasicEnergy} with our earlier results is enough to ultimately get an improvement on \eqref{eq:ClassicalLargeValue} in the key scenario $N = T^{4/5}$, $|W| = T^{3/5}$.  
The second bound is a little more complicated, but it leads to stronger estimates in our applications.

%
%

\begin{prpstn}[Bound for energy] \label{prp:energybound} Suppose that $D(t) = \sum_{n \sim N} b_n n^{it}$ with $|b_n| \le 1$.   Suppose that $W$ is a 1-separated set contained in an interval of length $T$,  and that $|D(t)| \ge N^\sigma$ for $t \in W$.   If $T^{3/4} \le N \le T$,  then

\begin{equation}
E(W)\lessapprox |W| N^{4-4\sigma}+|W|^{21/8}T^{1/4}N^{1-2\sigma}+|W|^3N^{1-2\sigma}.
\label{eq:EnergyBound}
\end{equation}

\end{prpstn}

%
%

Since the algebra is a little messy,  we take a moment to process the bounds.   For one thing,  if $E(W)$ is very large,  we get very strong bounds on $\sigma$.   For instance,  if $E(W) \approx |W|^3$,  and if  $N^{2/3} \le T \le N$, then Lemma \ref{lmm:BasicEnergy} gives a sharp estimate: either $N^\sigma \lessapprox N^{1/2}$ or  $|W| N^{2 \sigma} \lessapprox N^2$.   More important for our application is when we get an improvement on the basic orthogonality bound $|W| N^{2 \sigma} \ll T N$.   The full equations are a little messy,  but if we plug in the key scenario $\sigma=3/4$ and $N= T^{4/5}$, then Lemma \ref{lmm:BasicEnergy} gives an improvement on the bound $|W| \lessapprox T^{3/5}$ when $E(W) \ge |W|^{\frac{8}{3} + \epsilon}$ and Proposition \ref{prp:energybound} gives an improvement when $E(W) \ge |W|^{\frac{19}{8}+\epsilon}$ (for some $\epsilon > 0$).  

If $|W|\approx TN^{1-2\sigma}$, then the bound in Proposition \ref{prp:energybound} would be $(N/T)^2|W|^3+(|W|^{5/8}T^{-6/8})|W|^3 +|W|^4/T$. The first term will be the most important for us, and generally sets the limitations on our bounds. The second term will be negligible in practice (and could be improved with a bit more effort). The final term corresponds to the additive energy of a random set in an interval of length $T$.

The condition $N\ge T^{3/4}$ is used to simplify intermediate terms occurring in the proof of Proposition \ref{prp:energybound} and this range could be improved with some extra effort. For the purposes of Theorem \ref{thrm:LargeValues} the key situation is when $N=T^{5/6}$ as in Proposition \ref{prpstn:KeyProp}.

An immediate consequence of Proposition \ref{prp:energybound} is a good bound for the key term $S_3$ by substituting the bound of Proposition \ref{prp:energybound} (applied to $D_N(t)$) into Proposition \ref{prpstnS3}.

%
%

\begin{prpstn}[$S_3$ Bound]\label{prpstn:S3}
Let $N\ge T^{3/4}$. Then we have

\[
S_{3}\lessapprox T^2|W|^{3/2}+T|W|N^{3-2\sigma}+T|W|^2N^{3/2-\sigma}+T^{9/8}|W|^{29/16}N^{3/2-\sigma}.
\]

\end{prpstn}

Now we turn to the proof of Proposition \ref{prp:energybound}.

Suppose that $D(t) = \sum_{n \sim N} b_n n^{it}$ and $|D(t)| \ge N^\sigma$ on $W$.  Morally,  we also have $|D(t)| \gtrapprox N^\sigma$ if the distance from $t$ to $W$ is $\ll 1$.   In particular, if $|t_1 + t_2 - t_3 - t_4| \le 1$,  then morally $|D(t_1 + t_2 - t_3)| \gtrapprox N^\sigma$.   Therefore we should expect that

\[ E(W) \lessapprox \sum_{t_1, t_2, t_3 \in W} \frac{| D(t_1 + t_2 - t_3) |^2}{ N^{2 \sigma}}=N^{-2\sigma}\sum_{n_1,n_2\sim N}b_{n_1}\overline{b_{n_2}}R\Bigl(\frac{n_1}{n_2}\Bigr)^2R\Bigl(\frac{n_2}{n_1}\Bigr). \]

We can make this heuristic argument literally true by working with a slightly smoothed version of $D$.

%
%

\begin{lmm}[Dirichlet polynomials do not vary too fast] \label{lmm:randtrans} Let $D(t)$ be as in Proposition \ref{prp:energybound}. Then we have

\[
|D(t)|\ll \int_{|u-t|\lessapprox 1}|D(u)|du + O(T^{-100}).
\]

\end{lmm}

\begin{proof} 
Let $\psi(x)$ be a smooth bump which is supported on $|2\pi x-\log{N}|\ll 1$ and is equal to 1 on $[(2\pi)^{-1}\log{N},(2\pi)^{-1}\log{2N}]$, and extend $b_n$ to all $n\in \mathbb{Z}$ by setting $b_n=0$ if it is not the case that $n\sim N$. Then we have

\[
D(t)=\sum_{n\sim N}b_n n^{it}=\sum_{n}b_n n^{it} \psi\Bigl(\frac{\log{n}}{2\pi}\Bigr)=\int \widehat{\psi}(\xi) D(t-\xi)d\xi.
\]

By the rapid decay of $\widehat{\psi}$ we may restrict to $|\xi|\lessapprox 1$ at the cost of an error $O(T^{-100})$.
\end{proof}

%
%

\begin{lmm}[Energy controlled by discrete $3^{rd}$ moment]\label{lmm:Energy1}
Let $D(t)$ and $W$  be as in Proposition \ref{prp:energybound}. Then we have 

\[
E(W)\lessapprox N^{-2\sigma}\sum_{n_1,n_2\sim N}\Bigl|R\Bigl(\frac{n_1}{n_2}\Bigr)\Bigr|^3.
\]

\end{lmm}

\begin{proof}
Since $|D(t)|>N^\sigma$ for $t\in W$, we have

\[
E(W)=  \sum_{\substack{t_1, t_2, t_3, t_4 \in W\\  |t_1 + t_2 - t_3 - t_4| \le 1}}1\le N^{-2\sigma} \sum_{\substack{t_1, t_2, t_3, t_4 \in W\\  |t_1 + t_2 - t_3 - t_4| \le 1}}|D(t_4)|^2.
\]

By Lemma \ref{lmm:randtrans} and Cauchy-Schwarz, we have for $|t_1+t_2-t_3-t_4|\le 1$

\[
|D(t_4)|^2\ll \int_{|u-t_4|\lessapprox 1}|D(u)|^2 du \ll \int_{|u-t_1+t_2-t_3|\lessapprox 1}|D(u)|^2 du.
\]

Since $W$ is $1$-separated, given $t_1,t_2,t_3$ there are $\ll 1$ choices of $t_4\in W$ such that $|t_1+t_2-t_3-t_4|\le 1$. Thus we see that

\begin{align*}
E(W)&\ll N^{-2\sigma}\sum_{t_1,t_2,t_3\in W}\int_{s\lessapprox 1}|D(t_1+t_2-t_3-s)|^2ds\\
&= N^{-2\sigma}\sum_{n_1,n_2\sim N}b_{n_1}\overline{b_{n_2}}\int_{s\lessapprox 1}\Bigl(\frac{n_2}{n_1}\Bigr)^{is}R\Bigl(\frac{n_1}{n_2}\Bigr)^2 R\Bigl(\frac{n_2}{n_1}\Bigr)ds\\
&\lessapprox N^{-2\sigma}\sum_{n_1,n_2\sim N}\Bigl|R\Bigl(\frac{n_1}{n_2}\Bigr)\Bigr|^3.\qedhere
\end{align*}
\end{proof}

%
%

\begin{rmk}
In the final line of the proof of Lemma \ref{lmm:Energy1} we made important use of the assumption $|b_n|\le 1$. This is the only place in the paper where we make use of this $\ell^\infty$ bound rather than an $\ell^2$ bound. If we only had a weaker bound on the coefficients available, we would ultimately require a version of Heath-Brown's Theorem with weaker assumptions on the coefficients $b_n$.
\end{rmk}

%
%

Next we note that Heath-Brown's theorem (Theorem \ref{thrm:HeathBrown}) bounds $\sum_{n_1,n_2\sim N}|R(\frac{n_1}{n_2})|^2$ fairly directly and, with a bit more work, can be used to bound $\sum_{n_1,n_2\sim N}|R(\frac{n_1}{n_2})|^4$ too.

\begin{lmm}[Discrete second moment] \label{lmm:R2bound} For any $M \ge 1$, 

\[ \sum_{n_1,n_2\sim M}\Bigl|R\Bigl(\frac{n_1}{n_2}\Bigr)\Bigr|^2 \lessapprox  |W|^2 M + |W| M^2 + |W|^{5/4}T^{1/2} M. \]

\end{lmm}

\begin{proof}
We have that

\[
\sum_{n_1, n_2 \sim M}\Bigl|R\Bigl(\frac{n_1}{n_2}\Bigr)\Bigr|^2=\sum_{t_1,t_2\in W}\Bigl|\sum_{n\sim M} n^{i(t_1-t_2)}\Bigr|^2,
\]

so by Theorem \ref{thrm:HeathBrown} this is

\[
\lessapprox |W|^2 M + |W| M^2 + |W|^{5/4}T^{1/2} M.\qedhere
\]
\end{proof}

%
%

Lemma \ref{lmm:BasicEnergy} is now a quick consequence of our arguments so far.
\begin{proof}[Proof of Lemma \ref{lmm:BasicEnergy}]
By Lemma \ref{lmm:Energy1} and the trivial bound $|R(x)|\le |W|$ we have

\[
E(W)\lessapprox N^{-2\sigma}\sum_{n_1,n_2\sim N}\Bigl|R\Bigl(\frac{n_1}{n_2}\Bigr)\Bigr|^3\le |W|N^{-2\sigma}\sum_{n_1,n_2\sim N}\Bigl|R\Bigl(\frac{n_1}{n_2}\Bigr)\Bigr|^2.
\]

Lemma \ref{lmm:R2bound} (which is just Theorem \ref{thrm:HeathBrown}) now shows for $N>T^{2/3}$ we have

\[
E(W)\lessapprox |W|^3N^{1-2\sigma}+|W|^2N^{2-2\sigma}.\qedhere
\]
\end{proof}

To do better we look at higher moments to avoid the potentially wasteful use of the trivial bound $|R(x)|\le |W|$.

%
%

\begin{lmm}[Discrete fourth moment] \label{lmm:R4bound} For any $M \ge 1$, 

\[ \sum_{n_1,n_2\sim M}\Bigl|R\Bigl(\frac{n_1}{n_2}\Bigr)\Bigr|^4  \lessapprox  |W|^4 M + M^2 E(W) +  E(W)^{3/4}|W|T^{1/2}M.\]

\end{lmm}

\begin{proof} We split the sum in the $R$ function according to the number of representations of $u$ as approximately $t_1-t_2$. Let $\lfloor x \rfloor$ denote the largest integer $\le x$, and define

\[
U_B:=\Big\{u\in \mathbb{Z}:\, \# \{(t_1,t_2)\in W^2:\,\lfloor t_1-t_2\rfloor=u\} \sim B \Big\}.
\]

Clearly $U_B$ is empty if $B<1/2$ or if $B>|W|$. Thus, using Cauchy-Schwarz

\begin{align*}
|R(x)|^4=\Bigl|\sum_{t_1,t_2\in W}x^{i(t_1-t_2)}\Bigr|^2&=\Bigl|\sum_{B=2^j}\sum_{u\in U_B}\sum_{\substack{t_1,t_2\in W\\ \lfloor t_1-t_2\rfloor=u}}x^{i(t_1-t_2)}\Bigr|^2\\
&\lessapprox \sum_{B=2^j\le |W|}\Bigl|\sum_{u\in U_B}\sum_{\substack{t_1,t_2\in W\\ \lfloor t_1-t_2\rfloor=u}}x^{i(t_1-t_2)}\Bigr|^2.
\end{align*}

Taking $x=n_1/n_2$ and summing over $n_1,n_2\sim M$ then gives

\begin{align*}
\sum_{n_1,n_2\sim M}\Bigr|R\Bigl(\frac{n_1}{n_2}\Bigr)\Bigr|^4&\lessapprox \sup_{B\le |W|}\sum_{n_1,n_2\sim M}\Bigl|\sum_{u\in U_B}\sum_{\substack{t_1,t_2\in W\\ \lfloor t_1-t_2\rfloor=u}}\Bigl(\frac{n_1}{n_2}\Bigr)^{i(t_1-t_2)}\Bigr|^2\\
&\le\sup_{B\le |W|}\sum_{u_1,u_2\in U_B}\Bigl(\sum_{\substack{t_1,t_3\in W\\ \lfloor t_1-t_3\rfloor=u_1}}1\Bigr)\Bigl(\sum_{\substack{t_2,t_4\in W\\ \lfloor t_2-t_4\rfloor=u_2}}1\Bigr)\sup_{|s| \ll 1}\Bigl|\sum_{n\sim M}n^{i(u_1-u_2+s)}\Bigr|^2\\
&\lessapprox \sup_{B\le |W|}B^2\sum_{u_1,u_2\in U_B} \sup_{|s| \ll 1} \Bigl|\sum_{n\sim M}n^{i(u_1-u_2+s)}\Bigr|^2.
\end{align*}

By using Lemma  \ref{lmm:randtrans} to replace the supremum with an integral, and then applying Theorem \ref{thrm:HeathBrown},  we find

\begin{align*}
 \sum_{n_1,n_2\sim M}\Bigr|R\Bigl(\frac{n_1}{n_2}\Bigr)\Bigr|^4&\lessapprox \sup_{B\le |W|}B^2\int_{t\lessapprox 1}\sum_{u_1,u_2\in U_B} \Bigl|\sum_{n\sim M}n^{i(u_1-u_2+t)}\Bigr|^2dt\\
 &\lessapprox \sup_{B\le |W|}B^2\Bigl(|U_B|^2 M +  |U_B| M^2+ T^{1/2}|U_B|^{5/4} M\Bigr). 
 \end{align*}

We have that $B |U_B|\le |W|^2$ and $B^2 |U_B|\le E(W)$, so this gives

\begin{align*}
\sum_{n_1,n_2\sim M}\Bigr|R\Bigl(\frac{n_1}{n_2}\Bigr)\Bigr|^4&\lessapprox |W|^4 M + M^2 E(W) +  E(W)^{3/4}|W|T^{1/2}M.\qedhere
\end{align*}
\end{proof}

%
%

To bound $\sum_{n_1,n_2\sim N}|R\Bigl(\frac{n_1}{n_2}\Bigr)|^3$,  we could use H\"older:

\[ \sum_{n_1,n_2\sim N}\Bigl|R\Bigl(\frac{n_1}{n_2}\Bigr)\Bigr|^3 \le \left( \sum_{n_1,n_2\sim N}\Bigl|R\Bigl(\frac{n_1}{n_2}\Bigr)\Bigr|^2 \right)^{1/2} \left( \sum_{n_1,n_2\sim N}\Bigl|R\Bigl(\frac{n_1}{n_2}\Bigr)\Bigr|^4 \right)^{1/2}\]

\noindent and then bound the two factors using Lemmas \ref{lmm:R2bound} and \ref{lmm:R4bound}.   However, this H\"older step is somewhat lossy.   If $n_1'/n_2'$ is a rational number of small height,  then the sum $\sum_{n_1,n_2\sim N}\Bigl|R\Bigl(\frac{n_1}{n_2}\Bigr)\Bigr|^p$ counts $|R( n_1'/n_2')|^p$ many times -- because there are many $n_1,n_2 \sim N$ with $n_1/n_2 = n_1'/n_2'$.   The $4^{th}$ moment tends to be dominated by $n_1,n_2$ with large $\gcd(n_1,n_2)$, but the $2^{nd}$ moment tends to be dominated by $n_1, n_2$ with small $\gcd(n_1,n_2)$.   Therefore,  instead of doing H\"older immediately,  we now split our argument according to the size of $\gcd(n_1,n_2)$. 

Let $d=\gcd(n_1,n_2)$ and $n_1=n_1'd$, $n_2=n_2'd$ for some $n_1',n_2'\sim N/d$ with $\gcd(n_1',n_2')=1$. Thus we have for any choice of parameter $D$ (dropping the coprimality constraint when $d$ is large)
\begin{equation}
E(W)\le N^{-2\sigma}\sum_{d\le D}\sum_{\substack{n_1',n_2'\sim N/d\\ \gcd(n_1',n_2')=1}}\Bigl|R\Bigl(\frac{n_1'}{n_2'}\Bigr)\Bigr|^3+N^{-2\sigma}\sum_{d\ge D}\sum_{n_1',n_2'\sim N/d}\Bigl|R\Bigl(\frac{n_1'}{n_2'}\Bigr)\Bigr|^3.
\end{equation}

First we consider small $d$.   When $d$ is small enough, the distinct fractions $n_1'/n_2'$ are very well distributed and so it makes sense to compare our sum with $\int_{v\asymp 1} |R(v)|^3 dv$.  

We recall that $W$ is contained in an interval of length $T$.   Morally,  $|\widehat{W}(\tau)|$ is locally constant on intervals of length $1/T$.  Since $R(v) = \widehat{W}(\log v/(-2\pi))$,  we see that for $v \asymp 1$,  $|R(v)|$ is morally locally constant at scale $1/T$.  We make this precise in the following lemma:

%
%

\begin{lmm}\label{lmm:RStationary} For $v \asymp 1$,

\[ |R(v)| \ll T \int_{|v' - v| \lessapprox 1/T} |R(v')| dv' + O(T^{-100}). \]

\end{lmm}

\begin{proof} Since $v \asymp 1$,  we can do a change of variables,  $\tau =(- 2\pi)^{-1} \log v$,  and it suffices to prove that

\[ |\widehat{W}(\tau)| \ll T \int_{|\tau' - \tau| \lessapprox 1/T} |\widehat{W}(\tau')| d \tau' + O(T^{-100}). \]

We know that $W$ is contained in an interval of length $T$; call this $[T_0,T_0+T]$. Let $\psi$ be a smooth bump which is 1 on $[0,1]$. Then we have

\begin{align*}
\widehat{W}(\tau)=\sum_{t\in W}e(-t\tau)=\sum_{t\in W}e(-t\tau)\psi\Bigl(\frac{t-T_0}{T}\Bigr)=\int \widehat{\psi}(\xi)\widehat{W}\Bigl(\tau-\frac{\xi}{T}\Bigr)e\Bigl(\frac{-T_0\xi}{T}\Bigr)d\xi.
\end{align*}

By the rapid decay of $\widehat{\psi}$, we may restrict the integral to $\xi\lessapprox 1$ at the cost of an $O(T^{-100})$ error term. Since $\widehat{\psi}\ll 1$ this then gives the result.
\end{proof}

%
%

\begin{lmm}[Small GCD terms]\label{lmm:DSmall} We have
\[
\sum_{\substack{n_1,n_2\sim N\\ \gcd(n_1,n_2)\le D}}\Bigl|R\Bigl(\frac{n_1}{n_2}\Bigr)\Bigr|^3\lessapprox ( D T + N^2 ) |W|^{1/2} E(W)^{1/2}.
\]
\end{lmm}

\begin{proof}
Let $d=\gcd(n_1,n_2)$ and $n_1=n_1'd$, $n_2=n_2'd$ for some $n_1',n_2'\sim N/d$ with $\gcd(n_1',n_2')=1$. By Lemma \ref{lmm:RStationary}, we have

\[
\sum_{\substack{n_1',n_2'\sim N/d\\ \gcd(n_1',n_2')=1}}\Bigl|R\Bigl(\frac{n_1'}{n_2'}\Bigr)\Bigr|^3\ll T\int_{v\asymp 1} |R(v)|^3 \Bigl( \sum_{\substack{n_1,n_2'\sim N/d\\ \gcd(n_1',n_2')=1\\ |v - n_1'/n_2'| \lessapprox 1/T}}1\Bigr)dv.
\]

Since the fractions $n_1'/n_2'$ are $d^2/N^2$-separated, we have that the inner sum over $n_1',n_2'$ on the right hand side is $\lessapprox 1+N^2/(d^2T)$. Thus we find

\begin{align*}
\sum_{d\le  D} \sum_{\substack{n_1',n_2'\sim N/d\\ \gcd(n_1',n_2')=1}}&  \Big| R \Bigl(\frac{n_1'}{n_2'}\Bigr)\Big|^3 \lessapprox \sum_{d\le D}\Bigl(T+\frac{N^2}{d^2}\Bigr) \int_{v\asymp 1} |R(v)|^3dv \\
&\lessapprox \sum_{d\le D}\Bigl(T+\frac{N^2}{d^2}\Bigr) \left( \int_{v \asymp 1} |R(v)|^2 dv \right)^{1/2} \left( \int_{v \asymp 1} |R(v)|^4 dv \right)^{1/2} \\
&\lessapprox \Bigl( D T + N^2 \Bigr) |W|^{1/2} E(W)^{1/2}. \qedhere
\end{align*}
\end{proof}

%
%

We choose $D:=N^2/T$, so this gives 
\begin{equation}
\sum_{\substack{n_1,n_2\sim N\\ \gcd(n_1,n_2)\le D}}\Bigl|R\Bigl(\frac{n_1}{n_2}\Bigr)\Bigr|^3\lessapprox N^2|W|^{1/2}E(W)^{1/2}.\label{eq:SmallD}
\end{equation}

%
%

\begin{lmm}[Large GCD terms]\label{lmm:DBig}
Let $D=N^2/T$ and $N\ge T^{3/4}$. Then we have
\[
\sum_{\substack{n_1,n_2\sim N\\ \gcd(n_1,n_2)\ge D}}\Bigl|R\Bigl(\frac{n_1}{n_2}\Bigr)\Bigr|^3\lessapprox N|W|^3+NT^{1/4}|W|^{21/8}+E(W)^{1/2}|W|^{1/2}N^2.
\]
\end{lmm}
\begin{proof}
As in the previous lemma, we let $d=\gcd(n_1,n_2)$ and $n_1=n_1'd$, $n_2=n_2'd$. When $d$ is large, we keep the discrete summation over $n_1',n_2'$ and apply Cauchy-Schwarz directly, giving

\begin{align*}
\sum_{n_1',n_2'\sim N/d}\Bigl|R\Bigl(\frac{n_1'}{n_2'}\Bigr)\Bigr|^3 &\ll  \Bigl(\sum_{n_1',n_2'\sim N/d}\Bigr|R\Bigl(\frac{n_1'}{n_2'}\Bigr)\Bigl|^2\Bigr)^{1/2}\Bigl(\sum_{n_1',n_2'\sim N/d}\Bigl|R\Bigl(\frac{n_1'}{n_2'}\Bigr)\Bigr|^4\Bigr)^{1/2}.
\end{align*}

Now we can bound the factors on the right-hand side by Lemmas \ref{lmm:R2bound} and \ref{lmm:R4bound},  with $M= N/d$.   This gives

\begin{align*}
 \sum_{n_1',n_2'\sim N/d}\Bigl|R\Bigl(\frac{n_1'}{n_2'}\Bigr)\Bigr|^3&\lessapprox\left(   \frac{|W| N^2}{d^2} +\frac{|W|^2 N}{d}+\frac{|W|^{5/4}T^{1/2}N}{d}   \right)^{1/2} \\
&\times  \left( \frac{N|W|^4}{d}+\frac{N^2 E(W)}{d^2}+\frac{E(W)^{3/4}|W|T^{1/2}N}{d} \right)^{1/2}.
 \end{align*}

Summing over $d \ge D$, using Cauchy-Schwarz,  and recalling that $D = N^2/T$ then gives

\begin{align}
&\sum_{\substack{n_1,n_2\sim N\\ \gcd(n_1,n_2)\ge D}}\Bigl|R\Bigl(\frac{n_1}{n_2}\Bigr)\Bigr|^3 \lessapprox \Bigl(\frac{|W| N^2}{D}+|W|^2 N+|W|^{5/4}T^{1/2}N\Bigr)^{1/2}\nonumber\\
&\qquad\qquad\qquad\qquad\qquad\qquad \times\Bigl(N|W|^4+\frac{N^2 E(W)}{D}+E(W)^{3/4}|W|T^{1/2}N\Bigr)^{1/2}.\nonumber
\end{align}

Next we work on simplifying and organizing the algebra.  Recall that we have $N \ge T^{3/4}$ and $D=N^2/T$.   Therefore,  we have $|W|^{5/4} T^{1/2} N \ge |W| T=|W|N^2/D$,  and we can ignore the first term in the first factor. Thus the above expression is bounded by

\begin{equation}
 \lessapprox \Bigl(|W|^2N+|W|^{5/4}T^{1/2}N\Bigr)^{1/2}\Bigl(N|W|^4+E(W)T+E(W)^{3/4}|W|T^{1/2}N\Bigr)^{1/2}.\label{eq:BigDBound}
\end{equation}

There are two main cases, depending on whether $|W| > T^{2/3}$ or not.  If $|W|>T^{2/3}$ then $|W|^2N>|W|^{5/4}T^{1/2}N$, and so the first factor is dominated by $|W|^2 N$.   We turn to the second factor.  If $|W| > T^{2/3}$, then $N|W|^4>N|W|^{13/4}T^{1/2}\ge E(W)^{3/4}|W|T^{1/2}N$.   Also, since $N \ge T^{3/4} > T^{1/3}$ and $E(W)\le |W|^3$,  $N|W|^4>|W|^3T\ge E(W)T$.   So the second factor is dominated by $N|W|^4$.   Therefore, if $|W|>T^{2/3}$, \eqref{eq:BigDBound} simplifies to

\begin{align}
&\sum_{\substack{n_1,n_2\sim N\\ \gcd(n_1,n_2)\ge D}}\Bigl|R\Bigl(\frac{n_1}{n_2}\Bigr)\Bigr|^3\lessapprox N|W|^3.\label{eq:BigD1}
\end{align}

Now suppose $|W| \le T^{2/3}$.   We see that $|W|^2N\le |W|^{5/4}T^{1/2}N$, so the first factor is dominated by $|W|^{5/4} T^{1/2} N$.   Turning to the second factor, and recalling that $N \ge T^{3/4} > T^{1/2}$ and $E(W)\le |W|^3$,   we see that 
$E(W)^{3/4}|W|T^{1/2}N\ge E(W)T$. Thus, if $|W|\le T^{2/3}$ we see that \eqref{eq:BigDBound} simplifies to

\begin{align}
\sum_{\substack{n_1,n_2\sim N\\ \gcd(n_1,n_2)\ge D}}\Bigl|R\Bigl(\frac{n_1}{n_2}\Bigr)\Bigr|^3 &\lessapprox \Bigl(|W|^{5/4}T^{1/2}N\Bigr)^{1/2}\Bigl(N|W|^4+E(W)^{3/4}|W|T^{1/2}N\Bigr)^{1/2}\nonumber\\
&\lessapprox NT^{1/4}|W|^{21/8}+E(W)^{1/2}|W|^{1/2}N^2\Bigl(\frac{T^{1/2}|W|^{5/8}}{E(W)^{1/8}N}\Bigr).\label{eq:BigD2}
\end{align}

Since $E(W)\ge |W|^2$, $|W|\le T^{2/3}$ and $N\ge T^{3/4}$, we see that $T^{1/2}|W|^{5/8}\le E(W)^{1/8}N$, so the final term in \eqref{eq:BigD2} is $O(|W|^{1/2}E(W)^{1/2}N^2)$. Thus, combining \eqref{eq:BigD1} and \eqref{eq:BigD2}, we find that  provided $N\ge T^{3/4}$, regardless of the size of $W$, we have

\[
\sum_{\substack{n_1,n_2\sim N\\ \gcd(n_1,n_2)\ge D}}\Bigl|R\Bigl(\frac{n_1}{n_2}\Bigr)\Bigr|^3\lessapprox N|W|^3+NT^{1/4}|W|^{21/8}+E(W)^{1/2}|W|^{1/2}N^2.\qedhere
\]

\end{proof}

%
%

\begin{proof}[Proof of Proposition \ref{prp:energybound}]
First we use Lemma \ref{lmm:Energy1} to give

\[
E(W)\lessapprox N^{-2\sigma}\sum_{n_1,n_2\sim N}\Bigl|R\Bigl(\frac{n_1}{n_2}\Bigr)\Bigr|^3.
\]

Splitting according to whether $\gcd(n_1,n_2)\le D=N^2/T$ or not, we find by Lemma \ref{lmm:DSmall} and Lemma \ref{lmm:DBig} that

\[
E(W)\lessapprox N^{-2\sigma}\Bigl(N|W|^3+NT^{1/4}|W|^{21/8}+E(W)^{1/2}|W|^{1/2}N^2\Bigr).
\]

This rearranges to give

\[
E(W)\lessapprox |W|N^{4-4\sigma}+|W|^{21/8}T^{1/4}N^{1-2\sigma}+|W|^3N^{1-2\sigma}.\qedhere
\]
\end{proof}

%
%
%
%

\section{Proof of results on large values of Dirichlet polynomials}

In this section we prove our main results on the large values of Dirichlet polynomials by assembling the tools in the previous sections.

\begin{proof}[Proof of Proposition \ref{prpstn:KeyProp}]

Suppose that $|D_N(t)| \ge N^\sigma$ on the set $W$ contained in an interval of length $T = N^{6/5}$.  By Proposition \ref{prpstn:ImBound} and \eqref{eq:SiExpansion}, we have

\[
|W|\ll_\epsilon N^{2-2\sigma}+N^{1-2\sigma}\Bigl(\sum_{\substack{m\in \mathbb{Z}^3\setminus\{0\}}}I_m\Bigr)^{1/3}=N^{2-2\sigma}+N^{1-2\sigma}\Bigl(S_1+S_2+S_3\Bigr)^{1/3}.
\]

By Proposition \ref{prpstn:S1}, $S_1$ is negligible. We bound $S_2$ by Proposition \ref{prpstn:S2}, and $S_3$ by Proposition \ref{prpstn:S3}. Therefore we get for any choice of $k\in\mathbb{N}$
 
\begin{align*}
|W|^3N^{6\sigma-3}&\lessapprox_\epsilon N^{3} +S_2+  S_3\\
&\lessapprox_{\epsilon,k} N^3+|W|^2N^2+T N |W|^{2-1/k}+N^2 |W|^{2-3/4k}T^{1/2k}+T^2 |W|^{3/2}\\
&\qquad+T |W| N^{3-2\sigma}+T |W|^2 N^{3/2-\sigma}+T^{9/8} |W|^{29/16} N^{3/2-\sigma}.
\end{align*}

In this formula $k$ comes from the bound for $S_2$.   It is a positive integer that we can choose.   The last inequality rearranges to give

\begin{align}
|W|&\lessapprox_{\epsilon,k} N^{2-2\sigma}+N^{5-6\sigma}+T^{\frac{k}{k+1}}N^{(4-6\sigma)\frac{k}{k+1}}+N^{(5-6\sigma)\frac{4k}{4k+3}}T^{\frac{2}{4k+3}}+T^{4/3}N^{2-4\sigma}\nonumber\\
&\qquad +T^{1/2}N^{3-4\sigma}+T N^{9/2-7\sigma}+T^{18/19}N^{72/19-112\sigma/19}.\label{eq:FullBound}
\end{align}

We choose $k=4$.  We also simplify the formulas using $T = N^{6/5}$.   

\begin{align*}
|W|&\lessapprox_{\epsilon} T\Bigl(N^{(4-10\sigma)/5}+N^{(19-30\sigma)/5}+N^{(74 - 120\sigma)/25}+N^{(298-480\sigma)/95}\\
&\qquad\qquad+N^{(12 - 20\sigma)/5}+N^{(9-14\sigma)/2}+N^{(354-560\sigma)/95}\Bigr)\\
&\lessapprox_{\epsilon} TN^{(4-10\sigma)/5}+TN^{(12-20\sigma)/5}+TN^{(9-14\sigma)/2}.
\end{align*}

If $\sigma\in[7/10,8/10]$ the first and third terms can be dropped and we get

\[ |W| \lessapprox_{\epsilon} T N^{(12 - 20 \sigma)/5}. \qedhere
\]

\end{proof}

%
%

\begin{rmk}
When $N$ is not too close to $T$, the most important terms in our bound for $|W|$ in \eqref{eq:FullBound} are the terms $T^{4/3}N^{2-4\sigma}$ and $T^{1/2}N^{3-4\sigma}$, and our choice of $T=N^{6/5}$ comes from balancing these terms. In this range the terms coming from $S_2$ do not make a significant contribution and there is some flexibility with the choice of $k$; the $T^{\frac{k}{k+1}}N^{(4-6\sigma)\frac{k}{k+1}}$ term is typically increasing in $k$ and the $N^{(5-6\sigma)\frac{4k}{4k+3}}T^{\frac{2}{4k+3}}$ term is typically decreasing in $k$, so the optimal choice of $k$ balances these terms. The choice $k=4$ is made as a simple choice that is sufficient for the $S_2$ terms to be dominated by the $S_3$ terms in the range under consideration.
\end{rmk}

%
%

When $N>T^{5/6}$ the process of going from Proposition \ref{prpstn:KeyProp} to Theorem \ref{thrm:LargeValues} is somewhat wasteful since it actually bounds the number of large values in $[0,N^{6/5}]$. By using a variation of the above argument, the bound in Theorem \ref{thrm:LargeValues} could be improved. We record one such improvement here.

\begin{prpstn}[Large values estimate for $N\ge T^{5/6}$]\label{prpstn:BigN}
Suppose $(b_n)_{n\sim N}$, $(t_r)_{r\le R}$ are as in Theorem \ref{thrm:LargeValues}, and that $T^{5/6}\le N\le T$ and $V=N^\sigma$ with $\sigma\ge 7/10$. Then we have

\[
R\lessapprox N^{2-2\sigma}+T^{1/2}N^{3-4\sigma}+\inf_{k\in\mathbb{N}}\Bigl(T^{\frac{k}{k+1}}N^{(4-6\sigma)\frac{k}{k+1}}+N^{(5-6\sigma)\frac{4k}{4k+3}}T^{\frac{2}{4k+3}}\Bigr).
\]

In particular, we have

\[
R\lessapprox N^{2-2\sigma}+T^{1/2}N^{3-4\sigma}+T^{(30\sigma-21)/5}N^{(46-60\sigma)/5}.
\]

\end{prpstn}

Proposition \ref{prpstn:BigN} implies that the $N^{18/5}V^{-4}$ term in Theorem \ref{thrm:LargeValues} can be replaced by $T^{1/2}N^{3-4\sigma}+T^{(30\sigma-21)/5}N^{(46-60\sigma)/5}$, which is smaller. When $\sigma=3/4$, Proposition \ref{prpstn:BigN} improves on \eqref{eq:ClassicalLargeValue} by a factor of $(T/N)^{1/2}$.  Since we anticipate the main uses of Theorem \ref{thrm:LargeValues} to be when $N<T^{5/6}$ we content ourselves with the simpler formulation of Theorem \ref{thrm:LargeValues}.

\begin{proof}
Jutila's large values estimate \cite[Theorem (1.4)]{J} with $k=3$ gives

\[
R\lessapprox N^{2-2\sigma}+TN^{(10-16\sigma)/3}+TN^{18-24\sigma},
\]

which implies our bound for $\sigma\ge 39/50$ since the second and third terms above are then both smaller than $T^{1/2}N^{3-4\sigma}$. Thus we only need to consider $\sigma\in [7/10,39/50]$. The bound now follows from \eqref{eq:FullBound}, \eqref{eq:ClassicalLargeValue} and a little algebra. For $\sigma\in[7/10,39/50]$ and $N\in [T^{5/6},T]$ we have that

\[
T^{1/2}N^{3-4\sigma}\gg T^{4/3}N^{2-4\sigma}+T N^{9/2-7\sigma}+T^{18/19}N^{72/19-112\sigma/19}.
\]

Therefore the $T^{1/2}N^{3-4\sigma}$ term in \eqref{eq:FullBound} dominates the $5^{th}$, $7^{th}$ and $8^{th}$ terms in \eqref{eq:FullBound}. We also have

\[
T^{1/2}N^{3-4\sigma}\gg \min(TN^{1-2\sigma},N^{5-6\sigma}).
\]

Therefore, by combining \eqref{eq:FullBound} and \eqref{eq:ClassicalLargeValue} we find that

\[
R\lessapprox  N^{2-2\sigma}+T^{1/2}N^{3-4\sigma}+\inf_{k\in\mathbb{N}}\Bigl(T^{\frac{k}{k+1}}N^{(4-6\sigma)\frac{k}{k+1}}+N^{(5-6\sigma)\frac{4k}{4k+3}}T^{\frac{2}{4k+3}}\Bigr).
\]

This gives the first bound. For the final bound, we note that if $N^{6\sigma-4}\ge T$ then the result follows from \eqref{eq:ClassicalLargeValue}, so we may assume that $N^{6\sigma-4}<T$ and so the first term in the infimum above is increasing with $k$. Similarly, if $N^{4-4\sigma}>T$ then $TN^{1-2\sigma}<T^{1/2}N^{3-4\sigma}$ so the result follows from \eqref{eq:ClassicalLargeValue}. Therefore we may also assume that $N^{4-4\sigma}\le T$. In this case $N^{15-18\sigma}\le T^2$ and so in the second term in the infimum above is decreasing in $k$. Thus, for the expression to be less than $T^{(30\sigma-21)/5}N^{(46-60\sigma)/5}$ we require

\[
\frac{73 - 138 n - 90 \sigma + 180 n \sigma}{12(1-n)(10\sigma-7)}\le k \le \frac{-21 + 46 n + 30 \sigma - 60 n \sigma}{2(1-n)(13-15\sigma)}
\]

where $n:=\log{N}/\log{T}$. The upper bound of this interval is always at least 1 for $n\in[5/6,1)$, $\sigma\in [7/10,39/50]$ (it is increasing in $\sigma$ in this range so is smallest when $\sigma=7/10$) and the length of this interval is

\[
1+\frac{5(6n-5)(-41 + 123 \sigma - 90 \sigma^2)}{12(1-n)(10\sigma-7)(13-15\sigma)}.
\]

Thus the interval has length at least 1 whenever $\sigma\in [7/10,39/50]$, and so there is a choice of $k\in\mathbb{N}$ giving the desired bound.
\end{proof}

%
%
%
%

\section{Applications to Riemann zeta function and prime numbers}

\subsection{Proof of Theorem \ref{thrm:ZeroDensity}}

As noted in the introduction, Theorem \ref{thrm:ZeroDensity} follows from Ingham's result \eqref{eq:Ingham} if $\sigma\le 7/10$ and Huxley's result \eqref{eq:Huxley} if $\sigma\ge 8/10$, so we may assume that $\sigma\in[7/10,8/10]$. Clearly it suffices to show the bound of Theorem \ref{thrm:ZeroDensity} for zeros with imaginary part in $[T,2T]$, since the result for $[0,T]$ then follows by considering $T/2,T/4,\dots$ in place of $T$.

We now briefly recall the classical zero-detecting methodology, referring the reader to \cite[Chapter 12]{M3} or \cite[Appendix C]{MP} for more complete details. Given a parameter $N=2^j$, we let

\begin{align*}
b_n&:=\Bigl(\sum_{\substack{d|n\\ d\le 2T^{1/100}}}\mu(d)\Bigr)\exp\Bigl(-\frac{n}{T^{1/2}}\Bigr),\qquad D(s):=\sum_{n\sim N}b_n n^{-s}.
\end{align*}
(The reader should note that this is not quite consistent with the notation $D(t)$ used earlier in the paper.)

Then, a non-trivial zero $\rho=\beta+i\gamma$ of $\zeta(s)$ with $\beta\ge \sigma$ and $\gamma\in [T,2T]$ is called a `Type I zero' if there is a choice of $N=2^j\in [T^{1/100},T^{1/2}(\log{T})^2]$ such that $|D(\rho)|\ge 1/(3\log{T})$. If it is not a Type I zero then it is a `Type II zero', and the number of Type II zeros is $\le T^{2-2\sigma}(\log{T})^{O(1)}$ by \cite[Lemma 24]{MP}. Thus it suffices to bound the number of Type I zeros. There are $O(\log{T})$ choices of $N$ so we focus on the value of $N$ which gives the largest number of Type I zeros.

We now make a slight modification to $D$ to remove the dependencies on the real parts. Let $\rho=\beta+i\gamma$ be a Type I zero with $\beta\ge \sigma$, and let $\psi(u)$ be a smooth function equal to $e^{u(\sigma-\beta)}$ on $[\log{N},\log{2N}]$ and supported on $[(\log{N})/2,2\log{N}]$ with $\|\psi^{(j)}(t)\|_\infty \ll_j t^{-j}$ for all $j\in \mathbb{Z}_{\ge 0}$. We then note that by Fourier expansion

\[
D(\rho)=\sum_{n\sim N}b_n n^{-\sigma-i\gamma}\psi(\log{n})=\int_{\mathbb{R}}\widehat{\psi}(\xi)\Bigl(D(\sigma+i(\gamma-2\pi \xi))\Bigr)d\xi.
\]

Since $\widehat{\psi}$ is rapidly decreasing, we may truncate the integral to $\xi\lessapprox 1$ at the cost of an $O(T^{-100})$ error term. Therefore we see that if $\rho$ is a Type I zero, we have $|D(\sigma+i\gamma+i\xi)|\gtrapprox 1$ for some $\xi\lessapprox 1$. There are $O(\log{T})$ non-trivial zeros $\rho=\beta+i\gamma$ with $\gamma \in [t,t+1]$ for any $t\in [T,2T]$. Therefore we can find a 1-separated set of points $(s_r)_{r\le R}$ in $[T,2T]$ with $|D(\sigma+is_r)|\gtrapprox 1$ and the number $R$ of points satisfies $R\gtrapprox N(\sigma,2T)-N(\sigma,T)$. Let

\[
\tilde{b}_n:=\Bigl(\frac{N}{n}\Bigr)^\sigma b_n,\qquad \tilde{D}(t):=\sum_{n\sim N}\tilde{b}_nn^{it}=N^\sigma D(\sigma+it).
\]
Thus it suffices to show that if $N<T^{1/2+o(1)}$ and $W$ is a 1-separated set in $[T,2T]$ such that $|\tilde{D}(t)|\gtrapprox N^\sigma$, we have $|W|\lessapprox T^{15(1-\sigma)/(3+5\sigma)+o(1)}$.

First we note that if $N<T^{1/100}$ then $b_n=0$ identically, so the result is trivial in this range. Thus we may assume $N>T^{1/100}$. We can then choose a value of $k\ll 1$ such that

\begin{equation}
T^{10/(6+10\sigma)}\le N^k\le T^{15/(6+10\sigma)}.
\label{eq:kDef}
\end{equation}

This is clearly possible if $N\le T^{5/(6+10\sigma)}$, whereas if $N>T^{5/(6+10\sigma)}$ we can take $k=2$ since $N<T^{1/2+o(1)}$ and $15/(6+10\sigma)>1$ when $\sigma\le 8/10$. Set

\[
\alpha:=\frac{15(1-\sigma)}{(3+5\sigma)(18/5-4\sigma)}.
\]

If $N,k$ are such that $N^k\le T^\alpha$, then we apply Theorem \ref{thrm:LargeValues} to the Dirichlet polynomial $\tilde{D}^k$. This gives

\begin{align*}
|W|&\lessapprox N^{2k(1-\sigma)}+N^{(18/5-4\sigma)k}+TN^{(12/5-4\sigma)k}\lessapprox T^{15(1-\sigma)/(3+5\sigma)}.
\end{align*}

In the final inequality above we applied the upper bound from \eqref{eq:kDef} to the first term, the upper bound $N^k\le T^\alpha$ to the second term and the lower bound from \eqref{eq:kDef} to the final term.

If instead we have $N^k>T^\alpha$, then we apply the usual Mean Value Theorem to $\tilde{D}^k$. This gives

\begin{equation}
|W|\lessapprox N^{2k-2k\sigma}+TN^{k-2k\sigma}\lessapprox T^{15(1-\sigma)/(3+5\sigma)}+T^{1+(1-2\sigma)\alpha}.
\label{eq:WBound2}
\end{equation}

Here we applied the upper bound from \eqref{eq:kDef} to the first term and the lower bound from $N^k\ge T^\alpha$ to the second term. We then note that for $\sigma<9/10$ we have 
\[
1+(1-2\sigma)\alpha=\frac{129-195\sigma+50\sigma^2}{2(3+5\sigma)(9-10\sigma)}=\frac{15(1-\sigma)}{3+5\sigma}-\frac{250(\sigma-3/4)^2+3/8}{2(3+5\sigma)(9-10\sigma)},
\]
 and so \eqref{eq:WBound2} simplifies to $|W|\lessapprox T^{15(1-\sigma)/(3+5\sigma)}$, as required.

\begin{rmk}For the purposes of proving a zero density estimate of the form $N(\sigma,T)\lessapprox T^{A(1-\sigma)}$ with $A$ a fixed constant as small as possible, the critical case in our work is when $\sigma=7/10$, $N=T^{5/13}$ (so we apply Theorem \ref{thrm:LargeValues} to a Dirichlet polynomial of length $N^2=T^{10/13}$), we subdivide $[0,T]$ into intervals of length $T_1=T^{12/13}$ and where the set $W$ of large values on each subinterval has $|W|\approx T_1^{2/3}$ and $E(W)\approx |W|^{5/2}\approx |W|^4/T_1$. In this critical situation our bounds for $S_1$ and $S_2$ are both best possible,  and so any further improvement would have to come from the $S_3$ term. Our bound for $E(W)$ is also likely to be difficult to improve since a random set $W$ would have $E(W)\approx |W|^4/T_1$. The argument of Section \ref{sec:Affine} is also essentially tight if the $R$ function is taking $T_1^{2/3}$ values of size $|W|/M\approx T_1^{1/2}$ and the set of these values is highly concentrated on rationals with numerator and denominator of size $T_1^{1/3}$ (recall the remark after Proposition \ref{propsumaff}).
\end{rmk}

%
%

\subsection{Proof of Corollary \ref{crllry:All} and Corollary \ref{crllry:AlmostAll}}

These are well-known to follow quickly from \eqref{eq:ZeroDensity}, but for completeness we give a proof. By partial summation, it suffices to prove corresponding results for the Von Mangoldt function in place of the prime indicator function. By the explicit formula (see, for example \cite[Chapter 17]{D}) we have for any choice of $2\le T\le x$

\[
\sum_{n\in [x,x+y]}\Lambda(n)=y-\sum_{|\rho|\le T}\Bigl(\frac{(x+y)^{\rho}-x^\rho}{\rho}\Bigr)+O\Bigl(\frac{x (\log{x})^3}{T}\Bigr).
\]

We choose $T=xy^{-1}\exp(2\sqrt[4]{\log{x}})$ so the error term is $O(y\exp(-\sqrt[4]{\log{x}}))$, and note that the term in parentheses is $\int_x^{x+y}t^{\rho-1}dt\ll yx^{\Re(\rho)-1}$. Therefore, by considering $1/\log{x}$-separated values of $\sigma$, we find that 

\[
\sum_{n\in [x,x+y]}\Lambda(n)=y+O\Bigl(y (\log{x}) \sup_{\sigma}x^{\sigma-1} N(\sigma,T)\Bigr)+O(y\exp(-\sqrt[4]{\log{x}})).
\]

By combining \eqref{eq:ZeroDensity} with a slightly stronger result (such as \cite{J} or \cite[Theorem 12.1]{M3}) that loses at most logarithmic factors for $\sigma$ closer to 1, we have the bound

\begin{equation}
N(\sigma,T)\ll T^{(30/13+o(1))(1-\sigma)}(\log{T})^{O(1)}.
\end{equation}

Using this and the Vinogradov-Korobov zero-free bound $N(\sigma,T)=0$ for $\sigma \ge 1-c(\log{T})^{-2/3}(\log\log{T})^{-1/3}$ for a suitable constant $c>0$ (see \cite[Corollary 11.4]{M3}), we find that

\begin{align*}
 \sup_{\sigma}x^{\sigma-1} N(\sigma,T) &\ll (\log{T})^{O(1)} \sup_{\sigma \le 1-c(\log{T})^{-5/7}}\Bigl(\frac{T^{30/13+o(1)}}{x}\Bigr)^{1-\sigma}\\
 &\ll_\epsilon \exp(-\sqrt[4]{\log{x}})
\end{align*}

provided $T<x^{13/30-\epsilon/2}$. Recalling that $T=xy^{-1}\exp(2\sqrt[4]{\log{x}})$ and $y\ge x^{17/30+\epsilon}$, this gives Corollary \ref{crllry:All}.

For Corollary \ref{crllry:AlmostAll}, we first let $\delta=X^{-13/15+\epsilon/2}$. By splitting $[x,x+y]$ into intervals of length $\delta x$ and applying Cauchy-Schwarz, we see that

\begin{align*}
\int_{X}^{2X}\Bigl(\sum_{n\in [x,x+y]}\Lambda(n)-y\Bigr)^2 dx\ll  \frac{y^2}{\delta^2 X^2}\int_{X}^{3X}\Bigl(\sum_{n\in [x,x+\delta x]}\Lambda(n)-\delta x\Bigr)^2dx+O(\delta^2 X^{3}).
\end{align*}

If the corollary was false, the left hand side would be significantly larger than $y^2 X\exp(-3\sqrt[4]{\log{X}})$, so it suffices to show that the integral on the right hand side is $\ll \delta^2 X^3\exp(-3\sqrt[4]{\log{X}})$. Applying the explicit formula as above with $T=\delta^{-1}\exp(4\sqrt[4]{\log{X}})$, we see that it suffices to show that

\begin{equation}
\int_X^{3X} \Bigl| \sum_{|\rho|<T}x^\rho\Bigl(\frac{(1+\delta)^\rho-1}{\rho}\Bigr)\Bigr|^2dx\ll \delta^2 X^3 \exp(-3\sqrt[4]{\log{X}}).
\label{eq:AlmostAllTarget}
\end{equation}

Expanding the sum, and performing the integral over $x$, we obtain

\begin{align*}
&\sum_{|\rho_1|,|\rho_2|\le T}\Bigl(\frac{(1+\delta)^{\rho_1}-1}{\rho_1}\Bigr)\overline{\Bigl(\frac{(1+\delta)^{\rho_2}-1}{\rho_2}\Bigr)}\int_X^{3X}x^{\rho_1+\overline{\rho_2}}dx\ll \delta^2\sum_{|\rho_1|,|\rho_2|<T}\frac{X^{\Re(\rho_1)+\Re(\rho_2)+1}}{|\rho_1+\overline{\rho_2}+1|}.
\end{align*}

Since $X^{\Re(\rho_1)+\Re(\rho_2)+1}\le X^{2\Re(\rho_1)+1}+X^{2\Re(\rho_2)+1}$, and noting that (since there are $O(\log{T})$ zeros in a horizontal strip of height 1) we have

\[
\sum_{|\rho_2|<T} |1+z+\overline{\rho_2}|^{-1}\ll (\log{T})^2
\]

 for any $|z|<T$ with $\Re(z)\ge 0$. Thus we find that

\[
\int_X^{3X} \Bigl| \sum_{|\rho|<T}x^\rho\Bigl(\frac{(1+\delta)^\rho-1}{\rho}\Bigr)\Bigr|^2dx\ll (\log{X})^3 \delta^2\sup_{\sigma} X^{2\sigma+1}N(\sigma,T).
\]

As above, applying \eqref{eq:ZeroDensity} and the zero-free region, we have that

\[
\sup_{\sigma} X^{2\sigma+1}N(\sigma,T)\ll X^3\sup_{\sigma \le 1-c(\log{T})^{-5/7}}\Bigl(\frac{T^{30/13+o(1)}}{X^2}\Bigr)^{1-\sigma}\ll_\epsilon X^3\exp(-10\sqrt[4]{\log{X}}),
\]

on recalling that $T=\delta^{-1}\exp(4\sqrt[4]{\log{X}})\lessapprox X^{13/15-\epsilon/3}$. Putting this together then gives \eqref{eq:AlmostAllTarget}, as required.

\begin{rmk}
By using a prime decomposition (such as the Heath-Brown Identity) and Mellin inversion, it is possible to relate the count of primes in short intervals directly to Dirichlet polynomials. The critical situation for both Corollary \ref{crllry:All} and Corollary \ref{crllry:AlmostAll} is handling a product of six Dirichlet polynomials each of size roughly $x^{1/6}$ (this was the limiting case in the earlier work of Huxley \cite{Hu} too). As in \cite{HB4}, by bounding this contribution corresponding to six almost equal sized primes using a sieve method, one could obtain an asymptotic estimate in the slightly larger range $y\in[x^{17/30-\epsilon},x^{0.99}]$ and $y\in [X^{2/15-\epsilon},X^{0.99}]$ at the cost of a worse error term of size roughly $O(\epsilon^4 y/\log{x})$.
\end{rmk}

%
%
%
%

\bibliographystyle{plain}

\end{document}